\documentclass[12pt]{amsart}
\usepackage{amsmath,amsthm,amssymb}

\newcommand\mymat[4]{
{\left(
\begin{smallmatrix}#1&#2\\#3&#4\end{smallmatrix}
\right)}}

\newcommand\binomial[2]{
{
\left(
\begin{smallmatrix}#1\\#2\end{smallmatrix}
\right)
}}

\def\Sp{\operatorname{Sp}}

\def\SL{\operatorname{SL}}

\def\supp{\operatorname{supp}}
\def\Grit{\operatorname{Grit}}
\def\Borch{\operatorname{Borch}}
\def\Hum{\operatorname{Hum}}
\def\IM{\operatorname{Im}}
\def\FS{\operatorname{FS}}

\def\HA{{\mathcal H\/}}
\def\LA{{\mathcal L\/}}
\def\tr{\operatorname{tr}}
\def\Xn{{\mathcal X}_n}
\def\Xtwo{{\mathcal X}_2}
\def\Xnsemi{{\mathcal X}_n^{\text{\rm semi}}}

\def\R{{\mathbb R}}

\def\Z{{\mathbb Z}}

\def\C{{\mathbb C}}
\def\N{{\mathbb N}}
\def\Q{{\mathbb Q}}

\def\TT{{\mathcal T}}

\def\cal{\mathcal}

\theoremstyle{plain}
\newtheorem{thm}{Theorem}[section]
\newtheorem{lm}[thm]{Lemma}
\newtheorem{prop}[thm]{Proposition}
\newtheorem{cor}[thm]{Corollary}

\newtheorem{df}[thm]{Definition}

\newtheorem{conjecture}[thm]{Conjecture}
\theoremstyle{definition}

\def\supp{\mathop{\text{\rm supp}}}
\def\cusp{{\text{\rm cusp}}}
\def\ord{{\text{\rm ord}}}
\def\TB{\mathop{\text{\rm THBK}}}
\def\notdivides{\mathop{{\!\not\,|}}}

\def\SiegelH{{\cal H}}
\def\Half{{\cal H}}

\def\<{\langle}
\def\>{\rangle}

\def\wh{^\text{{\rm w.h.}}}
\def\weak{^\text{{\rm weak}}}
\def\mero{^\text{{\rm mero}}}
\def\inv{^{-1}}
\def\vp{\nu_{\text{poly}}}
\def\vs{\nu_{\text{series}}}
\def\vsg{\nu_{\text{Siegel}}}
\def\vJacobi{\nu_{\text{Jacobi}}}
\def\vJ{\nu_{\text{J}}}
\def\conv{\operatorname{conv}}
\def\cone{\operatorname{cone}}
\def\rcone{\operatorname{r-cone}}
\def\Cl{\operatorname{Closure}}
\def\rayA{{\vec A}}
\def\GG{\operatorname{{\mathcal G}}}
\def\pq{^{\text{\rm p.q.}}}
\newcommand\ldL{\llcorner}
\newcommand\rdL{\lrcorner}

\def\sing{\operatorname{sing}}
\def\Div{\operatorname{Div}}
\def\zeroIndicator{Z}
\def\Coeff{\operatorname{Coeff}}
\newcommand{\latt}[1]{{\langle{#1}\rangle}}
\def\sgn{\operatorname{sgn}}

\begin{document}
\title[Borcherds Product]
{ Borcherds Products Everywhere}

\author[V. Gritsenko]{Valery Gritsenko}
\address{Laboratoire Paul Painlev\'e, Universit\'e Lille 1 et
L'Institut universitaire de France
}
\email{Valery.Gritsenko@math.univ-lille1.fr}

\author[C. Poor]{Cris Poor}
\address{Dept{.} of Mathematics, Fordham University, Bronx, NY 10458, USA}
\email{poor@fordham.edu}

\author[D. Yuen]{David S. Yuen}
\address{Department of Mathematics and Computer Science, Lake Forest College, 555 N. Sheridan Rd., Lake Forest, IL 60045, USA}
\email{yuen@lakeforest.edu}

\date{\today}

\begin{abstract}
We prove the Borcherds Products Everywhere Theorem,
Theorem 6.6, that constructs holomorphic 
Borcherds Products from certain Jacobi forms that are theta blocks without 
theta denominator.  The proof uses generalized valuations from formal series to 
partially ordered abelian semigroups of closed convex sets.  
We present nine infinite families of paramodular Borcherds Products that are 
simultaneously Gritsenko lifts. This is the first appearance of infinite families with this property 
in the literature.
\end{abstract}


\maketitle
 

\section{Introduction}

\bigskip

This article studies Borcherds Products on 
groups that are simultaneously orthogonal and symplectic, 
the paramodular groups of degree two. 
This work began as an attempt to make Siegel paramodular cusp forms 
that are simultaneously Borcherds Products and Gritsenko lifts.  
On the face of it, this phenomenon may seem the 
least interesting type of a Borcherds product but it is the only known way to 
control the weight of a constructed Borcherds product.
Additionally, for computational purposes, 
a paramodular form that is both a Borcherds product and a Gritsenko lift is very useful; 
such a form has simple Fourier coefficients because it is a lift and a known divisor 
because it is a Borcherds product. 
In the case of weight~$3$, a Borcherds product gives
the canonical divisor class of the moduli space of $(1,t)$-polarized abelian surfaces.
Therefore  the construction of infinite families of such Siegel paramodular forms
is interesting for applications to algebraic geometry.
At the end of this article (see \S 8) we give 
nine infinite families of modular forms, including a family of weight~$3$, which 
are simultaneously Borcherds Products and Gritsenko lifts.
{\it This is the first appearance of such infinite families in the literature.}

All these Borcherds products are made by starting from certain  special Jacobi forms 
that are theta blocks without theta denominator.    
Theorem 1.1 gives a rather unexpected and surprising  
way to construct holomorphic Borcherds products starting from theta blocks of positive weight. 
As it is rather easy to search for theta blocks, 
we call this the  Borcherds Products Everywhere Theorem.  
The proof uses the theory of Borcherds products for paramodular forms 
as given by Gritsenko and Nikulin \cite{GritNiku98PartII}, 
the recent theory of theta blocks due to Gritsenko, Skoruppa and Zagier \cite{GSZ}, 
and a theory of generalized valuations on rings of formal series 
presented here in section~4.  
Let $\eta$ be the Dedekind Eta function and $\vartheta$ be the odd Jacobi theta function 
and write $\vartheta_{\ell}(\tau,z)=\vartheta(\tau, \ell z)$.  
The most general theta block  \cite{GSZ} can be written 
$\eta^{f(0)} \prod_{\ell \in \N} \left( { \vartheta_{\ell} }/{ \eta } \right)^{f(\ell)}$ 
for a sequence $f: \N \cup\{0\} \to \Z$ of finite support.  Here we consider only theta blocks {\it without theta denominator,\/} 
meaning that $f$ is nonnegative on $\N$.  
Theorem \ref{borcherdsproductseverywhere}  is a more detailed version of the main theorem 
but here is one suitable for this Introduction; 
the essential point is that the  Borcherds Products we construct are holomorphic, 
not just meromorphic.  

\begin{thm}\label{Thm1}
Let $v,k, t \in \N$.  
Let $\phi$ be a holomorphic Jacobi form of weight~$k$ and index~$t$ 
that is a theta block without theta denominator and that 
has vanishing order~$v$ in $q =e^{2 \pi i \tau}$.  
Then $\psi = (-1)^{v} \frac{ \phi | V_2 }{ \phi }$ 
is a weakly holomorphic Jacobi form of weight~$0$ and index~$t$ 
and the Borcherds lift of $\psi$ is a holomorphic paramodular form 
of level~$t$ and some weight $k' \in \N$.  
For even~$v$, it suffices for these conclusions that 
the theta block~$\phi$ without theta denominator  be weakly holomorphic. 
If ~$v=1$ then $k=k'$ and the first two Fourier Jacobi coefficients of the Borcherds lift of~$\psi$ and the 
Gritsenko lift of~$\phi$ agree.   
\end{thm}

 All Borcherds Products that are also Gritsenko lifts 
of theta blocks without theta denominator are necessarily generated in the manner of the 
above theorem but other holomorphic Borcherds Products can arise by the same process.  
In \cite{GritNiku98PartII}, Gritsenko and Nikulin point out that the leading Fourier 
Jacobi coefficient of a Borcherds Product, $\Borch(\Psi)$, is a theta block~$\phi$ 
and that when the Borcherds Product is also a Gritsenko lift we have 
$\Psi = - \frac{ \phi | V_2 }{ \phi }$.  
Gritsenko and Nikulin gave many examples of paramodular forms that are simultaneously  
multiplicative (Borcherds) and additive (Gritsenko) lifts, 
for both trivial and nontrivial characters of the paramodular group.  
In this article we consider only the case of 
trivial character.  Here we follow their line of thought and, beginning with a theta block~$\phi$  
without theta denominator 
that is a Jacobi form of positive vanishing order $v=\ord_q \phi$, show that 
$\Psi = (-1)^{v} \frac{ \phi | V_2 }{ \phi }$ is a weakly holomorphic weight zero Jacobi form with 
integral and positive singular part.  Hence,  $\Borch(\Psi)$ is a 
holomorphic paramodular form.  Proving that the character is trivial and determining the 
symmetry or antisymmetry require some combinatorics.

We note that the techniques developed to prove 
Theorem~{1.1} may be applied to construct antisymmetric Siegel paramodular 
forms that are new eigenforms for all the Hecke operators. These results will be presented in our next article. 
\smallskip

{\noindent {\bfseries Acknowledgements}:
We thank the authors of \cite{GSZ} for explaining their  work on theta blocks 
to the last two authors much in advance of its publication.   We thank 
F. Cl\'ery for helpful discussions.  The first author is grateful for financial support
under ANR-09-BLAN-0104-01 and Labex CEMPI.
The authors would like to thank Max-Planck-Institut
f\"ur Mathematik in Bonn for support and for providing excellent
working conditions during the activity  ``Explicit constructions in modular forms
theory".  

\section{Examples, especially of identities $\Grit(\phi)=\Borch\bigl(-(\phi|V_2)/\phi\bigr)$}  

The paramodular group $K(N)$ has a normalizing involution $\mu_N$
and 
a Borcherds product is a $\mu_N$ eigenform, see \S 3.   
In the following examples, 
we decompose $M_k(K(N))$ into a direct sum of $\mu_N$ eigenspaces.  
We write $M_k(K(N)) = M_k(K(N))^{+} \oplus M_k(K(N))^{-}$, 
where $M_k(K(N))^{\epsilon} = \{ f \in M_k(K(N)): f | \mu_N = \epsilon f\}$ 
and similarly for cusp forms.  

We will need the following criterion for cuspidality: 
For $k< 12$ and $p$~prime, 
elements of $M_k(K(p))^{\epsilon} $ whose Fourier expansion has the constant term zero 
are actually in $S_k(K(p))^{\epsilon} $.  
To give a proof, consider the Witt map $W: M_k(K(N)) \to M_k(\SL_2(\Z)) \otimes M_k(\SL_2(\Z))| \mymat{N}00{1}$, 
defined by $(Wf)(\tau, \omega)= f\mymat{\tau}00{\omega}$.  
Let $E_k \in M_k(\SL_2(\Z))$ be the Eisenstein series;  
for $k < 12$, the Witt image of~$f$ is $a\left( \mymat0000;f\right) E_k \otimes E_k|\mymat{N}00{1}$.  
An $f$ without a constant term 
satisfies $Wf=0$ and hence $\Phi(f)=0$, where $\Phi$ is Siegel's  map.  
For prime level~$p$, the only other $1$-cusp is represented by $\mu_p$ 
and so, when $f$ is also a $\mu_p$ eigenform, we have $\Phi(f | \mu_p )=0$ as well 
and $f$ is a cusp form.  In particular, since $M_2(\SL_2(\Z))=\{0\}$, 
we always have $M_2(K(p))= S_2(K(p))$.  
\vskip 0.2truecm

{\bf N=1.}  To construct  holomorphic Borcherds Products in $S_{k'}(\left( K(N) \right)$, 
we use a theta block 
$ \eta^u \prod_i \vartheta_{d_i}$ with $d_1^2 + \dots +d_{\ell}^2 = 2N$.  
In level one, for example, the only choice is $1^2+1^2=2$.  
Each $\eta$ contributes a vanishing of $q^{1/24}$ and each $\vartheta$ 
contributes $q^{3/24}$, so that $\phi_{10}= \eta^{18} \vartheta_1^2 \in J_{10,1}^\cusp$ 
has the lowest weight possible since the  vanishing order must be a  positive integer.
It is well known that $S_{10}(\Gamma_2) =\C \Psi_{10}$ 
is one dimensional and that Igusa's form $\Psi_{10} $ is the Saito-Kurokawa lift of 
$\phi_{10}$ as well as a Borcherds Product, $\Borch (\psi)$,
 which was found by Gritsenko and Nikulin  in \cite{GN1} where $\psi=-\left( \phi_{10} | V_2 \right)/ \phi_{10}
=20 +2\zeta+2\zeta^{-1}+\dots \in J_{0,1}\weak$.  Historically, this was the first 
example of Theorem \ref{Thm1}.

Constructing holomorphic Borcherds Products is usually considered a delicate task 
because the Fourier coefficients of the singular part of the lifted weakly holomorphic Jacobi form~$\psi$ 
must be mainly positive, but even in level one we can easily construct an infinite 
family of examples.  Set $\Delta = \eta^{24} \in S_{12}(\SL_2(\Z))$ and 
for $v \in \N$ 
define
$$
\phi_v = \Delta^{v-1} \eta^{18} \vartheta_1^2 \in J_{12v-2,1}^\cusp 
\text{ and } 
\psi_v= (-1)^v \frac{\phi_v | V_2}{\phi_v} \in J_{0,1}\wh  .  
$$
By Theorem~\ref{borcherdsproductseverywhere}, 
$\Borch(\psi_v) \in M_{k_v}(\Gamma_2)$ for some $k_v \in \N$.  
The second case is particularly interesting.  
The odd weight form  $\Borch(\psi_2)$ vanishes to order $89$ on the reducible locus, 
$\Hum\mymat0{1/2}{1/2}0=\Gamma_2\cdot\{z=0\}$, and so is divisible by $\Psi_{10}^{44}$; this leaves a form of weight~$35$.  
Therefore, we have $\Borch(\psi_2)=\Psi_{10}^{44} \Psi_{35}$ and a direct 
proof of the existence of a cusp form of weight~$35$ in level one.
The Borcherds product of $\Psi_{35}$ was found in \cite{GN2}.  
Table~1 presents results for small vanishing order~$v$.  
\smallskip

  Table 1.  Weight of $\Borch(\psi_v)$ and multiplicity on $\Hum \mymat0{1/2}{1/2}0$.  
\smallskip
 
\begin{tabular} {| r | r | r | } \hline
$v$   & $ k_v $ & $\text{Multiplicity on} \Hum \mymat0{1/2}{1/2}0 $   \\ \hline
{ }    & { } & { }  { } \\
$1$ & $ 10 $ & $  2$ \\
{ }    & { } & { }  { } \\
$2$ & $ 475$  &$89$ \\  
{ }    & { } & { }  { } \\
$3$ &$ 25228$  &   $ 4628$ \\
{ }    & { } & { }  { } \\
$4$ & $ 1409686$  &  $255902$ \\
{ }    & { } & { }  { } \\
$5$ & $ 81089336$  &    $ 14628136$  \\
{ }    & { } & { }  { } \\
$6$ & $ 4752949680$  &    $ 853836720$  \\
{ }    & { } & { }  { } \\
$7$ & $ 282277652800$  &    $ 50558528960$  \\
{ }    & { } & { }  { } \\
$8$ & $ 16928371578075$  &    $ 3025267676505$  \\
{ }    & { } & { }  { } \\
$9$ & $ 1022835157543260$  &    $ 182473970938500$  \\
{ }    & { } & { }  { } \\
$10$ & $ 62169320884762434$  &    $ 11075646070708830 $  \\
{ }    & { } & { }  { } \\
\hline
\end{tabular}
\bigskip

 The next series of examples  for $N=2$, $3$, $4$ and $5$ 
are related to reflective modular
forms whose divisors are determined by integral reflections in the corresponding 
projective stable integral orthogonal group 
$P\widetilde{O}^+(U\oplus U\oplus \latt{-2N})\cong PK(N)$
where $U$ is the hyperbolic plane, i.e., the even unimodular lattice of signature $(1,1)$. 
For the classification of all reflective paramodular forms see
\cite{GN3}. In particular,  in \cite{GritNiku98PartII} one finds the  analogues
of the Igusa modular forms $\Psi_{10}$ and $\Psi_{35}$ for 
the levels $N=2$, $3$ and $4$.  
\smallskip

{\bf N=2.}  For $K(2)$, we can pick  
$
\phi_{8,2} =  \eta^{12} \vartheta_1^4 \in J_{8,2}^\cusp 
$
and, setting
$
\psi_{8,2}= - \frac{\phi_{8,2} | V_2}{\phi_{8,2}} 
$, 
get a Borcherds Product with zero constant term, 
$\Borch(\psi_{8,2}) \in S_8(K(2))$, with representative singular part and divisor 
$$
\sing\left( \psi_{8,2} \right)=  16 +  4\zeta +4\zeta\inv, 
 \quad
\Div\left( \Borch(\psi_{8,2}) \right) = 4 \Hum \mymat0{1/2}{1/2}0 .
$$
We use the alternate notation, 
$\HA_N(r_o^2-4Nn_om_o, r_o)$,   
for Humbert surfaces 
$\Hum\mymat{n_o}{r_o/2}{r_o/2}{Nm_o}= 
K(N)^{+} \{ \mymat{\tau}zz{\omega} \in \Half_2: n_o\tau+r_oz+Nm_o\omega=0 \}$,  
see \cite{GritNiku98PartII} for details.  
Since $ \phi_{8,2}$ has order of vanishing~$v=1$, 
Theorem~\ref{borcherdsproductseverywhere} tells us that the first two Fourier Jacobi coefficients 
of $\Borch(\psi_{8,2})$ and $\Grit(\phi_{8,2})$ are equal.
Moreover, a part of the divisor of  the Gritsenko lifting is induced 
by the divisor of the lifted Jacobi form, see \cite{GritNiku98PartII}.   
Namely, if the Jacobi form has multiplicity~$m$ on $d$-torsion 
then the lift has multiplicity~$m$ on $\HA_N(d^2,d)$.  
Therefore, $\Div (\Grit(\phi_{8,2}))\supseteq \Div(\Borch(\psi_{8,2}))$ 
and the two forms coincide due to the Koecher principle.
This divisor argument together with the Witt map tells us that 
the space $S_{8}(K(2))$ is one dimensional.
In Ibukiyama and Onodera \cite{Ibuk97}, the ring structure of $M(K(2))$ was given and 
a generator $F_{8}$ of  $ S_{8}(K(2))$ was  constructed  as a polynomial in the 
thetanullwerte.  
Thus $\Grit(\phi_{8,2})=\Borch(\psi_{8,2})=F_8$ gives three very different constructions 
of the same modular form.   
Also, $ \phi_{11,2} =  \eta^{21} \vartheta_2 \in J_{11,2}^\cusp $ 
gives 
$\Borch(\psi_{11,2}) \in S_{11}(K(2))$, with representative singular part and divisor 
$$
\sing\left( \psi_{11,2} \right)=  22  + \zeta^2+ \zeta^{-2},
 \  
\Div\left( \Borch(\psi_{11,2}) \right) = \HA_2(4,2)+\HA_2(1,1).  
$$
Again, comparing the divisors and the first Fourier Jacobi coefficient, 
$\Borch(\psi_{11,2})$ and $\Grit(\psi_{11,2})$ are equal.  
Furthermore, 
$ S_{11}(K(2))=\C G_{11}$, for a generator~$G_{11}$ constructed from 
thetanullwerte.  The forms $F_8$ and $G_{11}$ are the cusp forms of lowest weight in 
the plus and minus spaces of the involution $\mu_2$, respectively, see  \cite{Ibuk97}.  
\smallskip

{\bf N=3.}  For $K(3)$, we can pick  
$
\phi_{6,3} =  \eta^{6} \vartheta_1^6 \in J_{6,3}^\cusp 
\text{ and } 
\psi_{6,3}= - \frac{\phi_{6,3} | V_2}{\phi_{6,3}} 
$ 
and get a Borcherds Product with a zero constant term in its Fourier expansion, $\Borch(\psi_{6,3}) \in S_6(K(3))$,  and with 
representative singular part and divisor 
$$
\sing\left( \psi_{6,3} \right)= 12  + 6\zeta+ 6\zeta\inv, 
 \quad
\Div\left( \Borch(\psi_{6,3}) \right) = 6 \HA_3(1,1).  
$$
As above, according to the divisor principle,  
we have $\Borch(  \psi_{6,3}  ) = \Grit( \phi_{6,3} )$.
The same argument shows that this is the only Siegel cusp form  for $K(3)$ of weight $6$,  
a fact first proved in \cite{Ibuk85}.   
Similarly, we have 
$
\phi_{9,3} =  \eta^{15} \vartheta_1^2\vartheta_2 \in J_{9,3}^\cusp 
$ 
and $\Borch(\psi_{9,3}) =\Grit(\phi_{9,3})$ spans $ S_9(K(3))$, with 
representative singular part and divisor 
\begin{align*}
\sing\left( \psi_{9,3} \right)  &= 18   + \zeta^2+ 2\zeta+ 2\zeta\inv + \zeta^{-2},  \\
\Div\left( \Borch(\psi_{9,3}) \right) &= \HA_3(4,2) +3 \HA_3(1,1) .  
\end{align*}

{\bf N=4.}   
For $K(4)$, 
$\phi_{4,4} =  \vartheta_1^8  \in J_{4,4}$ and 
$\Borch( \psi_{4,4}) \in M_{4}( K(4) )$ satisfy 
$$
\sing\left( \psi_{4,4} \right)  =  8  + 8\zeta + 8\zeta\inv, \quad
\Div\left( \Borch(\psi_{4,4}) \right) =  8\HA_4(1,1).  
$$
The Borcherds Product, $\Borch(\psi_{4,4})$,  is our first example of a noncusp form.  
The Jacobi form $\phi_{4,4}$ is not a cusp form but this does not affect the divisor 
argument. Thus $\Borch(\psi_{4,4})=\Grit(\phi_{4,4})$.
Also we have 
$\phi_{7,4} = \eta^{9} \vartheta_1^4\vartheta_2 \in J_{7,4}^\cusp$ and 
$\Borch( \psi_{7,4}) \in S_{7}( K(4) )$ satisfies  
\begin{align*}
\sing\left( \psi_{7,4} \right)  =  14  + \zeta^2 + 4\zeta + 4\zeta \inv + \zeta^{-2}, 
\\
\Div\left( \Borch(\psi_{7,4}) \right) =  \HA_4(4,2)+5\HA_4(1,1).  
\end{align*}
Also, 
$\phi_{10,4} = \eta^{18} \vartheta_2^2 \in J_{10,4}^\cusp$ and 
$\Borch( \psi_{10,4}) \in S_{10}( K(4) )$ satisfy 
\begin{align*}
\sing\left( \psi_{10,4} \right)  &= 20  + 2 \zeta^2 + 2 \zeta^{-2},  \\
\Div\left( \Borch(\psi_{10,4}) \right)   &=  2\HA_4(4,2)+2\HA_4(1,1).  
\end{align*}
In both cases the Gritsenko lift of $\phi$ is equal to the Borcherds Product for $\psi$ 
according to the divisor argument.
From \cite{IPY} we have the dimensions  
$\dim S_7(K(4)) =1$ and 
$\dim S_{10}(K(4)) =2$.  
\smallskip

{\bf N=5.}  (see \cite[\S 4.3]{GritNiku98PartII}).  
For $K(5)$, 
$\phi_{5,5} = \eta^{3} \vartheta_1^6\vartheta_2 \in J_{5,5}^\cusp$ and 
$\Borch( \psi_{5,5}) \in S_{5}( K(5) )$ satisfy 
\begin{align*}
\sing\left( \psi_{5,5} \right)  &= 10   + \zeta^2+ 6\zeta+ 6\zeta\inv + \zeta^{-2} ,   \\
\Div\left( \Borch(\psi_{5,5}) \right)  &=  \HA_5(4,2) +7 \HA_5(1,1).  
\end{align*}
Next, 
$\phi_{8,5} = \eta^{12} \vartheta_1^2\vartheta_2^2 \in J_{8.5}^\cusp$ and 
$\Borch( \psi_{8,5}) \in S_{8}( K(5) )$ satisfy 
\begin{align*}
\sing\left( \psi_{8,5} \right)  &= 16  + 2 \zeta^2+ 2 \zeta + 2 \zeta\inv + 2 \zeta^{-2}+ 2 q \zeta^5+ 2 q \zeta^{-5},  \\
\Div\left( \Borch(\psi_{8.5}) \right)  &= 2\HA_5(5,5) +2\HA_5(4,2) +4\HA_5(1,1) .  
\end{align*}
These Borcherds Products are cusp forms because their Fourier expansions have a constant term of zero.  
In both these cases we have the equality of $\Borch(\psi)$ and $\Grit(\phi)$.  
The first case follows from the divisor argument or the dimension 
$\dim S_5(K(5))=1$ from \cite{Ibuk85}.  The divisor argument does not apply to the 
second case because of the terms~$2 q \zeta^5+ 2 q \zeta^{-5}$; for this, we may refer ahead to section~8 
or  appeal to Table~{3} of \cite{IPY}, 
which shows that $S_k(K(5))^{\epsilon}$, for $\epsilon=(-1)^k$, is determined by the Fourier Jacobi coefficients 
of indices $5$ and $10$ for $k < 12$.   
\smallskip

A basis of reflective Jacobi forms for all possible $N$ was given in  \cite{GN3}.  
From these we obtain many other identities between  $\Borch(\psi)$ and $\Grit(\phi)$, 
the most important of which are the modular forms having multiplicity one on their divisors.  
In these cases, 
the Fourier expansion determines the generators and relations of a Lorentzian
Kac--Moody super Lie algebra. The next example is of a  different nature.  
\smallskip

{\bf N=37.} 
We turn to a favorite Jacobi form and prove something new about it.  
Let $J_{2,37}^\cusp= \C f$, where $f$ is the  Jacobi cusp form 
of weight two and smallest index
introduced in \cite{EZ}  where a table of its Fourier coefficients was given.  
We will prove that:
\begin{equation}
\label{Zagier}
\forall \, n,r \in \Z, \quad 
\sum_{ \alpha \in \Z}  
c\left( 6 \alpha^2 + n \alpha, 30 \alpha +r; f \right) =0.  
\end{equation}
In \cite{GSZ}, $f$ is shown to be the theta block 
$ \eta^{-6} \vartheta_1^3 \vartheta_2^3 \vartheta_3^2 \vartheta_4 \vartheta_5$.  
The vanishing order is one and by Theorem~\ref{borcherdsproductseverywhere}, 
setting $\psi = - ( f | V_2)/f$, 
we have a holomorphic Borcherds lift, 
$\Borch(\psi) \in S_2(K(37))$, that shares its first two Fourier Jacobi coefficients 
with $\Grit(f) \in S_2(K(37))$.  
In \cite{PoorYuenPara} it was shown that for primes $p< 600$, 
if $p \not\in \{ 277, 349, 353, 389, 461, 523, 587\}$ then the weight two paramodular cusp forms are spanned 
by Gritsenko lifts, that is, 
$S_2(K(p))=\Grit \left( J_{2,p}^\cusp \right)$.  
Thence $S_2(K(37))$ is one dimensional and 
we see that the Gritsenko lift of~$f$ is a Borcherds Product as well.  
The singular Fourier coefficients of~$\psi$ are represented by 
$$
\sing(\psi)= q^{6}\zeta^{30} + 4 + 3 \zeta + 3 \zeta^2 + 2\zeta^3 +\zeta^4 + \zeta^5. 
$$
Thus the divisor of $\Borch(\psi)=\Grit(f)$ is 
\begin{align*}
\Div\left( \Borch(\psi) \right) = 
  &\Hum\mymat{6}{15}{15}{37}
+ \Hum\mymat{0}{5/2}{5/2}{37}
+ \Hum\mymat{0}{2}{2}{37}            \\
+ 2  &\Hum\mymat{0}{3/2}{3/2}{37}
+ 4\Hum\mymat{0}{1}{1}{37}
+ 10\Hum\mymat{0}{1/2}{1/2}{37}.  
\end{align*}
Thus $\Grit(f)$ vanishes on the Humbert surface in $K(37)^{+} \backslash \Half_2$, 
$$
\Hum\mymat{6}{15}{15}{37} = K(37)^{+} 
\{ \mymat{\tau}{z}{z}{\omega} \in \Half_2: 6 \tau + 30 z + 37 \omega=0 \}.  
$$
In terms of Fourier coefficients, write 
$$
\Grit(f) \mymat{\tau}{z}{z}{\omega} = 
\sum_{ \alpha, \beta, \gamma \in \Z} 
a\left( \mymat{\gamma}{\beta/2}{\beta/2}{37 \alpha};  \Grit(f) \right) 
q^{\gamma} \zeta^{\beta} \xi^{37 \alpha},
$$
where $q=e(\tau)$, $\zeta= e(z)$ and $\xi = e(\omega)$.  
Substitution for $\xi$ using the relation $q^6 \zeta^{30} \xi^{37} =1$  gives
$$
\sum_{ \alpha, \beta, \gamma \in \Z } 
a\left( \mymat{\gamma}{\beta/2}{\beta/2}{37 \alpha};  \Grit(f) \right) 
q^{\gamma-6\alpha} \zeta^{\beta-30\alpha}  
=0, 
$$
or, setting $n=\gamma-6\alpha$ and $r= \beta-30 \alpha$, 
$$
\forall n,r \in \Z, \quad
\sum_{ \alpha \in \Z } 
a\left( \mymat{n+ 6 \alpha}{{(r+30\alpha)}/2}{{(r+30\alpha)}/2}{37 \alpha};  \Grit(f) \right)   
=0.   
$$
Using 
$a\left( \mymat{n}{r/2}{r/2}{Nm} ; \Grit(f) \right) = 
\sum_{\delta \in \N: \delta | (n,r,m)} \delta^{k-1}  
c\left( \frac{nm}{\delta^2}, \frac{r}{\delta}; f \right)$, 
for the Fourier coefficients of the Gritsenko lift, 
we obtain 
$$
\forall n,r \in \Z, \ \ 
\sum_{ \alpha \in \Z } \, 
\sum_{\delta \in \N:\, \delta | (n+6\alpha,\, r+30\alpha,\,  \alpha)} \delta 
c\left( \frac{(n+6\alpha)\alpha}{\delta^2}, \frac{r+30\alpha}{\delta}; f \right)  
=0.  
$$
This may be reduced, by induction on $\gcd(n,r)$, to the case $\delta=1$, which is equation~{(\ref{Zagier})}, 
as claimed.  A direct proof of equation~{(\ref{Zagier})} 
was shown to us by D. Zagier, and we leave this as a challenge to the reader. 
\smallskip

It would be desirable to have a direct proof of 
$\Borch(\psi) = \Grit(\phi)$ in general when 
the order of vanishing of the theta block $\phi$ is one, 
instead of relying on information about modular forms available in specific cases.  
Indeed  this, as presented in Conjecture~{\ref{wow}}, 
would be used as an additional tool in the investigation of modular forms.  
 We, in fact, prove Conjecture~{\ref{wow}} in many cases; namely, 
weights $4 \le k \le 11$.  
The examples of this current section for weights $k \ge 4$ 
are thus merely the first instances in infinite families of Borcherds Products 
that are also Gritsenko lifts.  
As weight one paramodular forms with trivial character vanish, 
this leaves only the cases of weights~$2$ and~$3$ open, 
see \S 8 for further examples and discussion.   

\section{Siegel Modular forms, Jacobi forms and liftings }

Let $\Sp_n(\R)$ act on the Siegel upper half space $\Half_n$ by linear 
fractional transformations.  
Let $V_n(\R)$ be the Euclidean space of real $n$-by-$n$ symmetric matrices with inner product 
$\langle A, B \rangle = \tr(AB)$, and extend this product $\C$-linearly to $V_n(\C)$.  
Let $\Gamma$ be a subgroup of projective rational elements of $\Sp_n(\R)$ 
commensurable with $\Gamma_n = \Sp_n(\Z)$.  We write 
$M_{k}(\Gamma, \chi)$ for the $\C$-vector space of Siegel modular forms 
of weight~$k$ and character~$\chi$ with respect to $\Gamma$.  
These are holomorphic functions $f: \Half_n \to \C$ that transform with respect to 
$\sigma= \mymat{A}BC{D} \in \Gamma$ by the factor of automorphy 
$\mu_{\text{Siegel}}^k \,\chi$ where $\mu_{\text{Siegel}}(\sigma, \Omega)= \det(C \Omega+D)$.  
Using the slash notation, 
$(f|_k \sigma)(\Omega)= \mu_{\text{Siegel}}(\sigma, \Omega)^{-k} f(\sigma \cdot \Omega)$, 
we have $f|_k \sigma = \chi(\sigma) f$ for all $\sigma \in \Gamma$.  
For $n=1$, we additionally require boundedness at the cusps, 
which is redundant for $n \ge 2$ by the Koecher Principle.  
The space of meromorphic~$f$ satisfying this automorphy condition is 
denoted by $M_{k}^{ \text{mero} }(\Gamma, \chi)$.   
The space of cusp forms is defined, using Siegel's $\Phi$ map, as 
$S_{k}(\Gamma, \chi)= \{ f \in M_{k}(\Gamma, \chi): 
\forall \sigma \in \Gamma_n, \Phi( f |_k \sigma)=0 \}$.    
A Siegel modular form has a Fourier expansion of the form 
$f(\Omega) = \sum_T a(T) e\left( \langle \Omega, T \rangle \right)$, 
where $T$ runs over $\Xnsemi(\Q)$, the  semidefinite elements of $V_n(\Q)$,  
and where $e(z)= e^{2 \pi i z }$.  
For cusp forms we may restrict the $T$ to $\Xn(\Q)$, the  definite elements of $V_n(\Q)$.  
The {\it principal\/} congruence subgroups are $\Gamma_n(N)= \{ \sigma \in \Gamma_n: \sigma \equiv I_{2n} \mod N \}$.  
We will be concerned with degree~$n=2$ and the {\it paramodular\/} group 
of level~$N$:  
$$
K(N)= 
\begin{pmatrix} 
*  &  N*  &  *  &  *  \\
*  &  *  &  *  &  */N  \\
*  &  N*  &  *  &  *  \\
N*  &  N*  &  N*  &  *  
\end{pmatrix} 
\cap \Sp_2(\Q), 
\quad \text{ $ * \in \Z$, }  
$$
which is isomorphic to the the integral symplectic group of the skew symmetric 
form with the elementary divisors $(1,N)$, see \cite{GritArith} and \cite{Grit2}.
Fourier expansions of paramodular forms sum over 
$T \in \Xtwo^{\text{\rm semi}}(N)= \{ \mymat{a}bb{Nc} \in \Xtwo^{\text{\rm semi}}(\Q): a, 2b, c \in \Z \}$.
 The paramodular group K(N) is not maximal in the real symplectic group 
$\Sp_2(\R)$ of rank $2$, see \cite{GritHulek0} for a complete description of its extensions.  
In particular,  
for any natural number $N>1$ the paramodular group $K(N)$ has a normalizing involution   
$\mu_N$ given by 
$\mu_N = \mymat{F_N^{*}}{0}{0}{F_N}$, where 
$F_N =\frac{1}{ \sqrt{N} } \mymat{0}{1}{-N}{0}$ is the Fricke involution, 
and we will frequently use the group 
$K(N)^{+}$ generated by $ K(N)$ and $\mu_N $.  
We let  $\chi_F: K(N)^+ \to \{ \pm 1 \}$ be the nontrivial character with kernel~$K(N)$
and observe that $M_k( K(N) ) = M_k(K(N)^+) \oplus M_k( K(N)^+, \chi_F)$ 
is the decomposition into plus and minus $\mu_N$-eigenspaces.  
\smallskip

The following definition of Jacobi forms,  see \cite{GritNiku98PartII},  
is equivalent to the usual one  \cite{EZ}. The only difference is that 
the book of Eichler and Zagier does not address Jacobi forms of half-integral
index,  which play a rather important role in \cite{GritNiku98PartII}.    
Consider two types of elements in  $\Gamma_2$,  
$$
h = 
\begin{pmatrix} 
1  &  0  &  0  &  v  \\
\lambda  &  1  &  v  &  \kappa  \\
0  &  0  &  1  &  -\lambda  \\
0  &  0  &  0  &   1
\end{pmatrix};
\quad 
\begin{pmatrix} 
a  &  0  &  b  &  0  \\
0  &  1  &  0  &  0  \\
c  &  0  &  d  &  0  \\
0  &  0  &  0  &   1
\end{pmatrix}, 
$$
for $ \lambda, v, \kappa \in \Z$,  and for $\mymat{a}bcd \in \SL_2(\Z)$.   
Let the subgroup of $\Gamma_2$ generated by the $h$ 
be called the Heisenberg group $H(\Z)$.  
The  character $v_H: H(\Z) \to \{ \pm 1 \}$ is 
defined by $v_H(h) = (-1)^{\lambda v + \lambda + v + \kappa}$.  
The second type constitute a copy of $\SL_2(\Z)$ inside $\Gamma_2$.  
This  copy of $\SL_2(\Z)$ and $H(\Z)$ and $ \pm I_4$ generate a group 
inside $\Gamma_2$ equal to
$$
P_{2,1}(\Z)= 
\begin{pmatrix} 
*  &  0  &  *  &  *  \\
*  &  *  &  *  &  *  \\
*  &  0  &  *  &  *  \\
0  &  0  &  0  &  *  
\end{pmatrix} 
\cap \Sp_2(\Z), 
\quad \text{ $ * \in \Z$.}  
$$
The character $v_H$ extends uniquely to a character on $P_{2,1}(\Z)$ 
that is trivial on the copy of $\SL_2(\Z)$.  Likewise, the factor of automorphy 
of the Dedekind Eta function
$$
\mu_{\eta}\left(  \mymat{a}bc{d}, \tau \right) = 
\dfrac{ \eta \left( \frac{a \tau +b}{c \tau + d} \right)}{ \eta(\tau) } = 
\sqrt{ c \tau + d }\, \epsilon\left(  \mymat{a}bc{d} \right),
$$
extends uniquely to a factor of automorphy on $ P_{2,1}(\Z) \times ( \Half_1 \times \C)$ 
that is trivial on $H(\Z)$ 
and we use this extension as the definition of the multiplier 
$\epsilon: P_{2,1}(\Z) \to e\left( \frac{1}{24}\Z \right)$.  
We also write $\epsilon = \epsilon_{\eta}$ for clarity.  

For $m \in \Q$, $a,b, 2k \in \Z$, 
consider holomorphic $\phi: \Half_1 \times \C \to \C$, 
such that the modified function ${\tilde \phi}: \Half_2 \to \C$,  
given by ${\tilde \phi}\mymat{\tau}{z}{z}{\omega} = \phi(\tau, z) e( m \omega)$, 
transforms by the factor of automorphy $\mu_{ \text{Siegel} }^k \epsilon^a v_H^b$ for $P_{2,1}(\Z)$.  
We always select holomorphic branches of roots that are positive on the purely imaginary elements of the Siegel half space.  
We necessarily have $2k \equiv a \equiv b \mod 2$ and $m \ge 0$ for nontrivial~$\phi$.  
Such $\phi$ have Fourier expansions 
$ \phi(\tau, z) = \sum_{ n,r \in \Q } c(n,r; \phi) q^n \zeta^r$, 
for $q=e(\tau)$ and $\zeta= e(z)$.  
We write $\phi \in J_{k,m}\wh( \epsilon^a v_H^b)$ if, 
additionally, the support of $\phi$ has $n$ bounded from below, 
and call such forms {\it weakly holomorphic.\/}   We write 
 $\phi \in J_{k,m}\weak( \epsilon^a v_H^b)$ if the support of 
 $\phi$ satisfies $n \ge 0$; 
 $\phi \in J_{k,m}( \epsilon^a v_H^b)$ if $4mn-r^2 \ge 0$; 
 $\phi \in J_{k,m}^\cusp( \epsilon^a v_H^b)$ if $4mn-r^2 > 0$.  
 Similar definitions are made for subgroups.  
 For example, $\phi \in J_{k,m}(\Gamma(N), \epsilon^a v_H^b)$ means that we demand 
 the automorphy of ${\tilde \phi}$ for $P_{2,1}(\Z) \cap \Gamma_2(N)$ and 
 demand the corresponding support condition at each cusp.

 For basic examples of Jacobi forms we make use of  
the Dedekind Eta function 
$
\eta(\tau)=q^{\frac{1}{24}} 
\prod_{n \in \N} (1-q^n) 
$
and the odd Jacobi theta function:  
\begin{align*}
\vartheta(\tau, z) 
&=\sum_{n \in \Z} (-1)^n q^{\frac{(2n+1)^2}{8}}\zeta^{\frac{2n+1}{2}}  \\
&=q^{\frac18}\left( \zeta^{\frac12}- \zeta^{-\frac12} \right) 
\prod_{j \in \N}(1 - q^j\zeta)(1-q^j\zeta^{-1})
(1-q^j).
\end{align*}
We have $\vartheta \in J_{\frac12,\frac12}^\cusp( \epsilon^3 v_H)$, 
$ \eta \in J_{\frac12,0}^\cusp( \epsilon)$ and 
$\vartheta_{\ell} \in J_{\frac12,\frac12 \ell^2}^\cusp( \epsilon^3 v_H^\ell)$, 
where $\vartheta_{\ell}(\tau,z) = \vartheta(\tau, \ell z)$ and $\ell \in \N$, 
compare \cite{GritNiku98PartII}.  
\smallskip

We now report on the existence and uniqueness of the characters, denoted  
$\epsilon^a \times v_H^b:K(N)^{+} \to e(\frac{\Z}{12})$, 
whose restriction to $P_{2,1}(\Z)$ is $\epsilon^a  v_H^b$ and whose value on $\mu_N$ is one.  
It follows from \cite{GritArith} that 
the extended paramodular group 
$K(N)^{+}$ is generated by $\mu_N$ and $ P_{2,1}(\Z)$.  
Thus any character $\chi: K(N)^{+} \to e(\Q)$ is determined by its value on $\mu_N$ 
and its restriction to $ P_{2,1}(\Z) $.  
For the existence, 
a result of \cite{GritHulek} is that the 
abelianization of $K(N)^{+}$ is $\Z/2\Z \times \Z / Q \Z$, 
where $Q= \gcd(2N,12)$.   
The character~$v_F$ is an element of order two.  
Furthermore, there is a character of order~$Q$ that has 
a restriction to $ P_{2,1}(\Z) $ given by $\epsilon^{24/Q} v_H^{2N/Q}$ 
and a value of one on $\mu_N$.  
Accordingly, for $a,b \in \Z$, the character $\epsilon^a \times v_H^b$ exists 
precisely when there is a $j \in \Z$ such that 
$
a \equiv j \frac{24}{\gcd(2N,12)} \mod 24$ and $b \equiv j \frac{2N}{\gcd(2N,12)} \mod 2.  
$
These brief considerations suffice,   
since we eventually prove that the paramodular forms considered here have  trivial character.  

\smallskip

Let $ \phi \in J_{k,t}\wh$ be a weakly holomorphic Jacobi form.  
Recall the level raising Hecke operators 
$V_{\ell}: J_{k,t}\wh \to J_{k,t\ell}\wh$ from \cite{EZ}, page 41.  
These operators stabilize both $J_{k,t}$ and $J_{k,t}^\cusp$ 
and have the following action on Fourier coefficients: 
$$
c\left( n,r;\phi | V_m \right)= 
\sum_{ d \in \N: \, d | (n,r,m) } 
d^{k-1} c\left( \frac{nm}{d^2},\, \frac{r}{d}; \phi \right).  
$$
Given any $\phi \in J_{k,t}\wh$, 
we may consider the following series  
$$
\Grit(\phi)\mymat{\tau}zz{\omega} = 
\delta(k)c(0,0;\phi)G_k(\tau)+
\sum_{m \in \N} 
\left( \phi | V_m \right)(\tau,z) e( mt \omega) 
$$
where  $\delta(k)=1$ for even  $k\ge 4$ and  $\delta(k)=0$ for all other $k$,  and  
$G_k(\tau)=(2\pi i)^{-k}{(k-1)!    \,   \zeta(k)}+
\Sigma_{n\ge 1}\sigma_{k-1}(n)e(\tau)$
is the Eisenstein series of weight $k$.  
\begin{thm} 
{\rm  (\cite{GritArith}, \cite{Grit2})\/} 
For $\phi \in J_{k,t}$, the series $\Grit(\phi)$ converges on $\Half_2$ and defines a holomorphic 
function $\Grit(\phi): \Half_2 \to \C$ 
that 
is an element of $M_k\left( K(t)^+, \chi_F^k \right)$.
This is a cusp form if  $\phi\in J^{\rm cusp}_{k,t}$.  
\end{thm}

The paramodular form $\Grit(\phi)$ is called the {\it Gritsenko lift\/} of the Jacobi form~$\phi$ 
and defines a linear map 
$\Grit: J_{k,t} \to M_k\left( K(t)^{+}, \chi_F^k \right)$.  
Forms with character $\chi_F^k$ are called {\it symmetric,\/}  
with character $\chi_F^{k+1}$, {\it antisymmetric.\/}   
Gritsenko lifts are hence symmetric.  
Antisymmetric forms are usually harder to construct. 
A different type of lifting construction is due to Borcherds (see \cite{B}) 
via his theory of infinite products 
in many variables on orthogonal groups.
The divisor of a Borcherds Product is supported on rational quadratic divisors.
In the case of the Siegel upper half plane of degree two, these 
rational quadratic divisors are the so-called 
Humbert modular surfaces.  
\begin{df}
Let $N \in \N$.  
For $n_o,r_o,m_o \in \Z$ with $m_o \ge 0$ and $\gcd(n_o,r_o,m_o)=1$,  
set $T_o = \mymat{n_o}{r_o/2}{r_o/2}{Nm_o}$ such that $\det(T_o) < 0$.
We call
$$
\Hum(T_o)= K(N)^{+} \{ \Omega \in \Half_2:\, \< \Omega, T_o \>=0 \} \subseteq K(N)^{+} \backslash \Half_2.
$$  
a {\sl Humbert modular surface\/}.  
\end{df}
From  \cite{GritHulek0} we have that a Humbert surface $\Hum(T_o)$ 
only depends upon two pieces of data: 
the discriminant $D=r_o^2-4Nm_on_o$ 
and $r_o \mod 2N$.  We may use this data to parameterize Humbert surfaces; 
write $\HA_N(D,r)= \Hum(T_o)$  
for any $T_o$ of the form $ \mymat{n_o}{r_o/2}{r_o/2}{Nm_o}$,  
with $\gcd(n_o,r_o,m_o)=1$ and $m_o \ge 0$, 
satisfying $-\det(2T_o) =D$ and $\langle T_o, \mymat0110 \rangle \equiv r \mod 2N$.  
For convenience, 
we extend the notation $\HA_N(D,r)$ to be empty when no such $T_o$ exists. 

The original Borcherds construction \cite{B} used the  Fourier coefficients of 
vector valued modular forms and was written using the Fourier expansion at a $0$-dimensional
cusp of an orthogonal modular variety.  A variant of Borcherds Products proposed   
by Gritsenko and Nikulin, see \cite{GN1} and \cite{GritNiku98PartII},  
was based on the Fourier expansion at a $1$-dimensional  cusp.  
The difference between these two approaches 
was explained in \cite{Grit24} 
for the  Borcherds modular form $\Phi_{12}\in M_{12}(O^+(II_{2,26}))$.  
For the proof of our main Theorem \ref{borcherdsproductseverywhere} we will use 

\begin{thm} 
\label{BP}
{\rm (\cite{GN1}, \cite{GritNiku98PartII}, \cite{Grit24})}   
Let $N, N_o \in \N$.    
Let $\Psi \in J_{0,N}\wh$ be a weakly holomorphic Jacobi form with Fourier expansion
$$
\Psi(\tau,z) = \sum_{n,r \in \Z: \, n \ge -N_o } c\left( n,r \right) q^n \zeta^r
$$
and $c(n,r) \in \Z$ for $4Nn-r^2 \le 0$.  
Then we have  $c(n,r) \in \Z$ for all $n,r \in \Z$.  
We set
\begin{align*}
&24A= \sum_{\ell \in \Z} c(0,\ell); \quad 
2 B = \sum_{\ell \in \N} \ell c(0,\ell); \quad 
4 C = \sum_{\ell \in \Z} \ell^2 c(0,\ell);   \\
&D_0 = \sum_{n \in \Z: \, n <0} \sigma_0(-n) c(n,0); \ 
k= \frac12 c(0,0); \ 
\chi = (\epsilon^{24A} \times v_H^{2B}) \chi_F^{k+D_0}.
\end{align*}
There is a function $\Borch(\Psi) \in M_k\mero\left( K(N)^+, \chi \right)$ 
whose divisor in $K(N)^{+} \backslash \Half_2$ consists of Humbert surfaces 
$\Hum(T_o)$ 
for $T_o = \mymat{n_o}{r_o/2}{r_o/2}{Nm_o}$ with $\gcd(n_o,r_o,m_o)=1$ and $m_o \ge 0$.  
The multiplicity of $\Borch(\Psi) $ on $\Hum(T_o)$ is 
$\sum_{n \in \N} c(n^2n_om_o, nr_o)$.  
In particular, 
if $c(n,r) \ge 0$ when $4Nn-r^2 \le 0$ then 
 $\Borch(\Psi) \in M_k\left( K(N)^+, \chi \right)$.  In particular,  
$$
\Borch(\Psi)(\mu_N\latt{\Omega})=(-1)^{k+D_0} \Borch(\Psi)(\Omega), 
\text{ for $\Omega \in \Half_2$. }  
$$ 
For sufficiently large  $\lambda$, for $\Omega=\mymat{\tau}zz{\omega} \in \Half_2$ 
and $q=e(\tau)$, $\zeta=e(z)$, $\xi=e(\omega)$, 
the following product converges on $\{\Omega \in \Half_2: \IM \Omega > \lambda I_2 \}$:  
$$
\Borch(\Psi)(\Omega){=}
q^A \zeta^B \xi^C 
\prod_{\substack{ n,r,m \in \Z:\, m \ge 0, \text{\rm\ if $m=0$ then $n \ge 0$} \\  \text{\rm and if $m=n=0$ then $r < 0$. } }}
\left( 1-q^n \zeta^r \xi^{Nm} \right)^{ c(nm,r) }
$$
and is on $\{\Omega \in \Half_2: \IM \Omega > \lambda I_2 \}$ 
a rearrangement of
$$
\Borch(\Psi)= 
\left(  \eta^{c(0,0)}  \prod_{ \ell \in \N} \left( \frac{{\tilde \vartheta}_{\ell}}{\eta} \right)^{c(0,\ell)} \right) 
\exp\left( - \Grit({ \Psi })\right).  
$$
\end{thm}

{\bf Remarks:} This last representation of $\Borch(\Psi)$ 
gives an experimental algorithm for the construction of Borcherds products.  
It gives  the first two
Fourier Jacobi coefficients of $\Borch(\Psi)$:
the first one is a theta block $\Theta= \eta^{ c(0,0)} \prod (\vartheta_{\ell}/\eta)^{c(0,\ell)}$ 
and the second is the product
$-\Theta \Psi$.  
As is standard, the convergence of an infinite product on $\Half_2$ is 
not defined to mean that the sequence of partial products has a limit;  
rather, it means that for each $\Omega \in \Half_2$, 
some tail of the product has a sequence of partial products with a nonzero limit.   
The next proposition is essential in the proof of the Theorem just stated, see   \cite{GritNiku98PartII}.

\begin{prop} 
\label{Identity}
Continuing with the notation of Theorem~\ref{BP}, set 
$D_1=\sum_{n ,r\in\Z:\, n<0\,}\sigma_1(-n)\,c(n,r;\Psi)$.  
We have $tA-tD_1-C=0$.  
\end{prop}

By using multiplicative Hecke operators, 
one can show that paramodular Borcherds Products satisfy special identities, 
see Theorem~{3.3} and the identity~{(3.25)} in   \cite{GritNiku98PartII}.  
Heim and Murase have proven a converse, that these special multiplicative 
identities in fact  characterize  Borcherds Products 
among automorphic forms, see \cite{HeimMurase}.  

\section{Generalized Valuations}  

Let $R$ be a ring and $G$ an abelian semigroup.  
A map $ \nu: R \setminus \{0\} \to G$ 
satisfies the {\it valuation property\/} on $R$ if 
$$
\nu(fg) = \nu(f) + \nu(g)
$$
for all nontrivial $f, g \in R$;  
we call such $\nu$ a {\it generalized valuation, \/}    
When $G$ is also partially ordered, one could potentially 
ask for the additional property: 
for all $y \in G$, $y \ge \nu(f)$ and $y \ge \nu(g)$ imply $y \ge \nu(fg)$, 
but this property plays no direct role for us here.   What is important is that 
certain rings of formal 
series admit generalized valuations into partially ordered abelian semigroups of closed convex sets.  

For a simple example of a generalized valuation, 
consider the ring $R_n = \C[x_1, x_1\inv, \dots, x_n, x_n\inv]$ of 
Laurent polynomials in $n$~variables.  
Let $f= \sum_{I \in \Z^n} a(I) x^{I} \in R_n$, 
where we use the multi-index notation $x^I=\prod_{i=1}^{n} x_i^{I_i}$,  
and define the support of $f$ as 
$\supp(f) =\{ I \in \Z^n: a(I) \ne 0 \}$.  
For $S \subseteq \R^n$, denote the {\it convex hull\/} of~$S$ by 
$\conv(S) = 
\{ \sum_{i=1}^{\text{finite}} \alpha_i s_i  \in \R^n: 
s_i \in S, \, 
\alpha_i \ge 0, \, \sum_{i=1}^{\text{finite}} \alpha_i =1\}$.  
Define the generalized valuation 
\begin{align*}
\vp: R_n \setminus \{0\} & \to \text{ Closed convex subsets of $\R^n$,}  \\
f  & \mapsto \conv\left( \supp(f)  \right). 
\end{align*}
The closed convex subsets of~$\R^n$ are a partially ordered abelian semigroup under 
pointwise addition 
$K_1+K_2= \{ x+y \in \R^n:  x \in K_1 \text{ and } y \in K_2 \}$ 
and inclusion $ K_1 \subseteq K_2$.  The map $\vp$ is indeed a generalized 
valuation and selected examples will convince the reader that this is not altogether trivial.  

\begin{prop}
Let $f,g \in R_n$ be nontrivial Laurent polynomials.  We have 
$
\vp(fg) = \vp(f) + \vp(g).  
$
\end{prop}
\begin{proof}
We have $\supp(fg) \subseteq \supp(f) + \supp(g)$ 
directly from the definition of polynomial multiplication and taking convex hulls gives one containment 
$\vp(fg) \subseteq \conv\left( \supp(f) + \supp(g) \right) =  \vp(f) + \vp(g)$.  
The other containment uses the Krein-Milman Theorem: 
a compact set $K$ in a Euclidean space $V$ is the convex hull of its extreme points $E(K)$; 
an extreme point of~$K$ being, by definition, 
a point of~$K$ that is not in the interior of any line segment contained in~$K$, see \cite{Rock}, page 167.  
Let $I_o$ be an extreme point of 
$ \vp(f) + \vp(g)= \conv\left( \supp(f) + \supp(g) \right) $ 
so that necessarily $I_o \in \supp(f) + \supp(g) $, see \cite{Rock}, page 165.  
We will conclude the proof by showing that $I_o \in \supp(fg)$ so that 
\begin{align*}
\vp(fg)   &=  \conv\left(  \supp(fg)  \right)  \\ 
               & \supseteq \conv\left(  E(  \vp(f) + \vp(g) )  \right) 
               \overset{K.M.}{=}  \vp(f) + \vp(g).  
\end{align*}

If $I_o \not\in \supp(fg)$, 
then the coefficient of $x^{I_o}$ was cancelled in the multiplication of~$f$ and~$g$ and so 
there are at least two decompositions 
$$
I_o = A+B= a+b,
$$
with $A,a \in \supp(f)$ and $B,b \in \supp(g)$ and $(A,B) \ne (a,b)$, hence $A \ne a$ and $B \ne b$.  
Let $m,w \in  \vp(f) + \vp(g)$ be defined as $m =A+b$ and $w =a +B$.  
Note that $m$, $w$ and $I_o$ are distinct.  
However, $I_o$ is the midpoint of $ \overline{mw} \subseteq  \vp(f) + \vp(g)$ 
and $I_o$ is extreme, a contradiction.  
\end{proof}
The ``other'' property is $ \vp(f+g) \subseteq \conv\left(  \vp(f) \cup \vp(g) \right)$.  
\smallskip

A naive generalization of the valuation property to infinite series is false.  
For example, $(1+x+x^2+\dots )(1-x)=1$ but 
$[0, \infty) + [0,1]=[0, \infty) \ne \{0\}$.  
From the point of view of convex geometry, 
the issue here is applying an appropriate generalization of the 
Krein-Milman Theorem to closed convex subsets~$C$ of a euclidean space~$V$.  
Let $P \subseteq V$ be a subset called {\it points\/} and $D \subseteq V$ be a 
subset called {\it directions.\/}   Define a generalized notion of convex hull 
that incorporates directions as well as points by 
\begin{align*}
&\conv(P; D) = \\
&\{ \sum_{i=1}^{\text{finite}} \alpha_i p_i + \sum_{j=1}^{\text{finite}} \beta_j v_j \in V: \, 
p_i \in P, \, v_j \in D, \, 
\alpha_i, \beta_j \ge 0, \, \sum_{i=1}^{\text{finite}} \alpha_i =1\}.
\end{align*}
More geometrically, we can say 
$ \conv(P; D) = \conv(P) + \cone_0(D)$, 
where $\cone(D)= \R_{> 0} \conv(D)$ and  $\cone_0(D)= \R_{\ge 0} \conv(D)$ as usual.  
A {\it face\/} of a convex set~$C$ is a convex subset~{$\tilde C$} of $C$ such that every closed line 
segment in $C$ with a relative interior point in $\tilde C$ lies entirely in $\tilde C$.  
Thus, the extreme points of $C$ are exactly the zero dimensional faces.  
If a face~$\tilde C$ of~$C$ is a ray then the parallel ray from the origin is called 
an {\it extreme direction\/.}  The {\it recession cone\/} of a nonempty convex set~$C$ 
is defined as 
$$
\rcone(C)=\{y \in V: C + \R_{\ge 0}y \subseteq C \} .  
$$
For a nonempty closed convex set, 
an extreme direction is necessarily a subset of the recession cone, see \cite{Rock}, page~163 
or Theorem~{8.3} on page 63.  
For an important example, consider the convex parabolic region 
$C=\{(x,y) \in \R^2: y \ge x^2\}$;  the recession cone of $C$ consists of the 
nonnegative $y$-axis, which is {\sl not\/} an extreme direction of $C$.  
Here is a generalization of the Krein-Milman Theorem which allows unboundedness.  

\begin{thm}
\label{GKM}
{\rm (\cite{Rock}, page 166.)}  
Let $V$ be a Euclidean space.  
Let $C \subseteq V$ be a closed convex set containing no lines.  
Then $C$ is the convex hull of its extreme points and extreme directions.  
\end{thm}

For a second example, 
define a generalized valuation $\vs$ on polar Laurent series of one variable
\begin{align*}
\vs: \C((x)) \setminus \{0\} & \to \text{ Closed convex subsets of $\R^1$,}  \\
f  & \mapsto \conv\left( \supp(f); [0, \infty)  \right). 
\end{align*}
The valuation property of $\vs$ follows from the equality 
$\vs(f)=[\min(\supp(f)), \infty)$, 
so that the left endpoint is the usual ``order of vanishing'' valuation of polar Laurent series.  
In general finite dimension, the defintion
\begin{align*}
\vs: \C((x_1, \dots,x_n)) \setminus \{0\} & \to \text{ Closed convex subsets of $\R^n$,}  \\
f  & \mapsto \conv\left( \supp(f); [0, \infty)^n  \right), 
\end{align*}
satisfies the valuation propetry, 
as can be proven directly or reduced to the one variable case.  
Here, the other property is  $\vs(f+g) \subseteq \conv\left( \vs(f) \cup \vs(g); [0,\infty)^n \right)$.  
\smallskip

For a third example, 
similar results hold for Siegel modular forms.  
For $f \in S_k(\Gamma_n)$ with 
$\supp(f)= \{ T \in \Xn: a(T;f) \ne 0 \}$ and 
$\vsg=\Cl_{V_n(\R)}\left(  \conv\left( \R_{\ge 1} \supp(f) \right) \right)$, 
we have the valuation property for nontrivial cusp forms, see \cite{PSY}.  
As a remark, $\vsg(f)$ has the following property, see \cite{PY00}.  
If $\det(Y)^{k/2} | f( X + i Y) | $ attains its maximum at $X_o+iY_o \in \Half_n$, 
then $\frac{k}{4 \pi} Y_o\inv \in \vsg(f)$.  
\smallskip

Generalized valuations for Jacobi forms require extra care, 
being an intermediate case between elliptic and Siegel modular forms.  
For $N \in \N$,  let 
$R(N) = \C[ \zeta^{1/N}, \zeta^{-1/N}]((q^{1/N}))$ 
be formal polar Laurent-Puiseaux series in two variables, 
considering~$q$ to be the first and~$\zeta$ the second variable.  
For $f = \sum_{n,r \in \Q} c(n,r) q^n \zeta^r \in R(N)$ 
define the support $\supp(f)=\{ (n,r) \in (\frac1{N} \Z)^2: c(n,r) \ne 0 \}$ and 
$$
\vJ(f)= 
\vJacobi(f) = \Cl_{\R^2}
\left(  \conv\left( \supp(f); [0,\infty) \times \{0\} \right) \right).
$$
The map $\vJ$ does not have the valuation property 
on the entire ring $R(N)$ 
but we regain the valuation property on a subring that allows recession only in the $q$-direction.  
Set the ray $\rayA= [0,\infty) \times\{0\}$ for brevity. 
We will require that $\supp(f)$ be contained in a closed convex parabolic region 
with recession cone~$\rayA$.  
\begin{prop}
The set $R_J(N)= 
\{ f \in R(N): \exists a \in \R_{+}, \exists b,c \in \R: 
\supp(f) \subseteq 
\{ (n,r) \in \R^2:  n \ge a r^2 +br+ c \}  \}  
$
is a ring.  
\end{prop}
\begin{proof}
This just amounts to the fact that the sum or union of 
convex parabolic regions with the same recession cone is contained in another such region.  
\end{proof}
In $R_J(N)$, the closure in the definition of $\vJ$ is redundant.  
\begin{lm}
\label{star}
For nontrivial $f \in R_J(N)$, 
we have 
$E\left( \vJ(f) \right) \subseteq \supp(f)$ and 
\begin{align*}
\vJ(f) &=  \Cl_{\R^2}\left(  \conv \left( \supp(f); \rayA \right) \right) 
=     \conv\left( \supp(f); \rayA \right) .
\end{align*}
\end{lm}
\begin{proof}
Let $I_o$ be an exposed extreme point of $\vJ(f)$; 
we will show that $I_o \in \supp(f)$.  
A face~${\tilde C}$ of a convex set~$C$ is {\it exposed\/} if 
${\tilde C}$ is the subset of~$C$ where some linear function~$h$ attains its minimum on~$C$.  
Let $h$ be the linear function on $\R^2$ that attains its minimum on $\vJ(f)$ uniquely at 
$I_o=(n_o,r_o)$.  We have $h(n,r)= \alpha n + \beta r$ for some $\alpha$, $\beta \in \R$.  
First, $ \alpha \ge 0$ since $(n_o,r_o) + \rayA \subseteq \vJ(f)$ and, additionally, 
$\alpha >0$ since 
$\{ I \in \vJ(f): h(I) = h(I_o) \}= \{ I_o\}$.  
Because $\alpha >0$ and $\supp(f)$ is contained in the latttice $(\frac1{N} \Z)^2$ 
and $\supp(f)$ is contained in a convex parabolic region with recession cone $\rayA$, 
the set 
${\mathcal S}(B)=\{ I \in \supp(f): h(I) \le B \}$ is finite for any $B >0$.  
Therefore $\inf_{I \in \supp(f) } h(I) = \min_{ I \in {\mathcal S}(h(I_1)) } h(I)$ 
for any $I_1 \in \supp(f)$.  
Thus $\rho= \min_{I \in \supp(f) } h(I)$ exists and is attained on a finite set ${\mathcal M} \subseteq \supp(f)$.  
It follows that $h \ge \rho$ on $\conv(\supp(f); \rayA)$ and that, 
by the continuity of $h$, we have $h \ge \rho$ on the closure $\vJ(f)$.  
Thus, the minimum of~$h$ on $\vJ(f)$ equals~$\rho$ and 
the face of $\vJ(f)$ where $h$ attains its minimum~$\rho$ equals $\{ I_o\}$ 
and contains ${\mathcal M\/}$;  
therefore $\{I_o\}={\mathcal M} \subseteq \supp(f)$.  
Thus the exposed extreme points of $\vJ(f)$ are contained in $\supp(f)$.  
By a theorem of Straszewicz, \cite{Rock}, page 167, 
any extreme point of a closed convex set is the limit of exposed extreme points.  
Since $\supp(f)$ is contained in a lattice, it is its own closure and we have 
$E\left( \vJ(f) \right) \subseteq \supp(f)$, which is the first assertion of this lemma.  
The extreme directions of $\vJ(f)$ are contained in $\{ \rayA \}$ 
and $\vJ(f)$ contains no lines,  
so that by the generalized Krein-Milman Theorem~\ref{GKM} we obtain 
$$
\vJ(f) = \conv\left( E\left( \vJ(f) \right); \rayA \right) 
\subseteq \conv \left( \supp(f); \rayA \right),
$$
which proves the second assertion.  
\end{proof}

\begin{thm}
\label{Good}
For  
$\vJ: R_J(N) \setminus \{0\}  \to \text{Closed convex subsets of $\R^2$}$,  
the valuation property 
$\vJ(fg) = \vJ(f) +\vJ(g)$ holds 
for all nontrivial $f,g \in R_J(N)$.  
\end{thm}
\begin{proof}
Since  $\supp(fg) \subseteq \supp(f) + \supp(g)$, 
we use Lemma~\ref{star} to conclude 
\begin{align*}
\vJ(fg) & = \conv\left(  \supp(fg); \rayA \right)  \\
            & \subseteq \conv\left(  \supp(f) + \supp(g) ; \rayA \right)  = \vJ(f) +\vJ(g).
\end{align*}
To prove $\vJ(f) +\vJ(g) \subseteq \vJ(fg) $,  begin by 
taking an extreme point~$I_o$ of 
$ \vJ(f) +\vJ(g) = \conv\left(  \supp(f) + \supp(g) ; \rayA \right) $.   
We necessarily have $I_o \in  \supp(f) + \supp(g)$ by 18.3.1 of  \cite{Rock}, page 165.  
As in the proof for polynomials, $I_o \in \supp(fg)$ so that 
$E \left(  \vJ(f) +\vJ(g) \right) \subseteq \supp(fg)$.  
Since $\vJ(f) +\vJ(g)$ contains no lines,  
the generalized Krein-Milman Theorem~\ref{GKM} gives 
\begin{align*}
\vJ(f) +\vJ(g) & = \conv\left(  E \left(  \vJ(f) +\vJ(g) \right)  ; \rayA \right)  \\
            & \subseteq \conv\left(  \supp(fg)  ; \rayA \right)  = \vJ(fg).
\end{align*}
\end{proof}

\begin{lm}
\label{crislemma} 
The ring $R_J(N)$ contains 

\begin{itemize}
\item[1.] the polar Laurent polynomials, $\C[\zeta^{1/N}, \zeta^{-1/N}, q^{1/N}, q^{-1/N}]$, 
\item[2.]  
 the Fourier expansions of weakly holomorphic Jacobi forms of level~$N$, 
$ J_{k,m}\wh\left( \Gamma(N) \right) $, 
\item[3.] 
the  infinite  products of the form 
$\prod_{j=1}^{\infty} \left( 1 + q^{j/N} h_j( q^{1/N}, \zeta^{1/N}) \right)$, 
where the $h_j \in \C[q, \zeta, \zeta\inv]$ are (either trivial or) 
polynomials of uniformly bounded degree; that is, there exists a $D>0$ such that for all $j$, 
$\deg_{\zeta} h_j = \max \{ |r|: \exists n \in \Q: (n,r) \in \supp(h_j) \} \le D$.  
\end{itemize}
\end{lm}
\begin{proof}
For item~1, 
the support $\supp(f)$ of a polynomial~$f$ is compact and so is contained in 
some convex parabolic region with recession cone~$\rayA$.  

For item~2 and positive index~$m$, 
a weakly holomorphic Jacobi form $ \phi \in J_{k,m}\wh\left( \Gamma(N) \right) $ 
has $4mn-r^2$ bounded below for $(n,r) \in \supp(\phi)$ 
and thus the Fourier expansion of $\phi$ is in $R_J(N)$.  
Index~$m=0$ is a degenerate case.  Here the Fourier expansion depends only upon $q$ and 
$\vJ(f)$ is a ray on the $q$-axis bounded from below.   

For the infinite product in item~3, consider an $(n,r)$ in the support.  
This requires at least $N \frac{|r|}{D}$ factors of the form $q^{j/N} h_j( q^{1/N}, \zeta^{1/N})$, 
which means the power of~$q$ is at least 
$\sum_{j=1}^{N|r|/D} j/N =\frac12 N \frac{|r|}{D} ( N \frac{|r|}{D}+1)\frac1{N}$;  
therefore $n \ge \frac{N}{2D^2} r^2$ and the support is contained in a 
convex parabolic region with recession cone~$\rayA$.  
\end{proof}

\begin{lm}
\label{davelemma}
For an infinite product 
$\prod_{j=1}^{\infty} \left( 1 + q^{j/N} h_j( q^{1/N}, \zeta^{1/N}) \right)$ 
as in item~3 of Lemma~\ref{crislemma}, we have
$$
\vJ\left( \prod_{j=1}^{\infty} \left( 1 + q^{j/N} h_j \right) \right)
= 
\bigcup_{m=1}^{\infty}
\vJ\left( \prod_{j=1}^{m} \left( 1 + q^{j/N} h_j \right) \right).  
$$
\end{lm}
\begin{proof}
To prove ``$\supseteq$,'' 
take any $m\in\N$ and 
any $I\in \vJ(\prod_{j=1}^m(1 + q^{j/N} h_j )$.
Then because $0\in\vJ(\prod_{j=m+1}^\infty(1 + q^{j/N} h_j))$,
we have
$$I\in\vJ(\prod_{j=1}^m(1 + q^{j/N} h_j))+
\vJ(\prod_{j=m+1}^\infty(1 + q^{j/N} h_j)),$$
which by Theorem~\ref{Good} and Lemma~\ref{crislemma}, items~1 and~3, 
implies $I\in\vJ(\prod_{j=1}^\infty( 1 + q^{j/N} h_j))$.

Next we prove ``$\subseteq$''.
Take any $(n,r)\in\supp(\prod_{j=1}^\infty(1 + q^{j/N} h_j))$.
Then it must be that 
$(n,r)\in\supp(\prod_{j=1}^{Nn}(1 + q^{j/N} h_j))$
since the higher factors cannot contribute to a $q^n \zeta^r$ term.  Thus 
$$
\supp\left( \prod_{j=1}^{\infty} \left( 1 + q^{j/N} h_j \right) \right) 
\subseteq 
\bigcup_{m=1}^{\infty}
\vJ\left( \prod_{j=1}^{m} \left( 1 + q^{j/N} h_j \right) \right)
$$
and so 
\begin{align*}
\vJ\left( \prod_{j=1}^{\infty} \left( 1 + q^{j/N} h_j \right) \right) 
&\subseteq 
\conv \left( \bigcup_{m=1}^{\infty}
\vJ\left( \prod_{j=1}^{m} \left( 1 + q^{j/N} h_j \right) \right); \rayA \right)  \\ 
&=\bigcup_{m=1}^{\infty}
\vJ\left( \prod_{j=1}^{m} \left( 1 + q^{j/N} h_j \right) \right)
\end{align*}
since the valuation property for $\vJ$ on Laurent polynomials shows, as above, that the 
$\vJ\left( \prod_{j=1}^{m} \left( 1 + q^{j/N} h_j \right) \right)$ are nested and since each has recession cone $\rayA$.
\end{proof}
\smallskip

We conclude this section with a few remarks and some notation.  
Let $\FS(\phi)$ denote the Fourier series of a weakly holomorphic Jacobi form 
and write $\vJ(\phi)$ for $\vJ( \FS(\phi))$.  
For a $ \phi \in J_{k,m}\wh$, the generalized valuation $\ord$ defined in \cite{GSZ} 
by 
$$\ord(\phi;x)= \min_{ (n,r) \in \supp(\phi) } (mx^2 + rx + n)$$  
is related to $\vJ(\phi)$ by 
$\ord(\phi;x)= \min  \< \vJ(\phi), \mymat{1}xx{x^2} \>$.  
This gives a variant proof of the valuation property of $\ord$ proven in \cite{GSZ}.  

For $ \phi \in J_{k,m}^\cusp$, 
let $(\tau_o,z_o) \in \Half_1 \times \C$ be the point where the invariant function 
$v^{k/2} e^{-2 \pi m y^2/v } | \phi(u+iv, x+iy) |$ 
attains its maximum.  
One can use the techniques of \cite{PY00} to prove that 
$\ord(\phi; x) \le \frac{k}{4 \pi} \frac{1}{v_o}+ m (x- \frac{y_o}{v_o}  )^2$.  

\section{Theta Blocks}  

Theta blocks are the invention of V. Gritsenko, N.-P. Skoruppa and D. Zagier, see \cite{GSZ} 
for a full treatment.  We will only cite the properties we need here.  
A {\it theta block\/} is a function of the form 
$$
\eta^{f(0)} \prod_{\ell \in \N} \left( \frac{ \vartheta_{\ell} }{ \eta } \right)^{f(\ell)}
$$
for some sequence $f: \N \cup \{0\} \to \Z$ with finite support.  
Reference to Theorem~\ref{BP} shows that theta blocks arise naturally 
as the leading Fourier Jacobi coefficient of a paramodular Borcherds Product.  

A theta block is a meromorphic Jacobi form with easily calculable 
weight, index, multiplier and divisor.  The generalized valuation~$\ord$, 
introduced in \cite{GSZ}, is also simple to calculate on theta blocks.  
Let $\GG=C^0(\R/ \Z)\pq$ be the additive group of continuous functions 
$g: \R \to \R$ that have period one and are piecewise quadratic.  
Define the positive (non-negative) elements in $\GG$ to be the semigroup of functions 
whose values are all positive (non-negative) in $\R$; this makes $\GG$ a partially 
ordered abelian group.  

For $\phi \in J_{k,m}(\chi)\wh$ and $x \in \R$ define
$$
\ord(\phi;x)= \min_{ (n,r) \in \supp(\phi) } (n+rx + mx^2).
$$
Then $\ord:   J_{k,m}(\chi)\wh \to \GG$, defined by $\phi \mapsto \ord(\phi)$ 
is a generalized valuation in the sense that it satisfies  
$$
\ord(\phi_1 \phi_2) = \ord(\phi_1) + \ord(\phi_2)
$$
on the ring of all weakly holomorphic Jacobi forms and 
$$
\ord(\phi_1+ \phi_2) \ge  \min\left(  \ord(\phi_1) , \ord(\phi_2)  \right)
$$
on each graded piece of fixed weight,  index and multiplier.  
The generalized valuation characterizes Jacobi forms from among weakly holomorphic Jacobi forms 
because a weakly holomorphic Jacobi form~$\phi$ is a Jacobi form if and only if 
$\ord(\phi) \ge 0$, and is a Jacobi cusp form if and only if  $\ord(\phi) > 0$.  
One can easily test to see whether a theta block~$\phi$  is a Jacobi form by checking the 
positivity of $\ord(\phi)$, which has the following pleasant formula  \cite{GSZ} 
$$
\forall x \in \R, \ 
\ord \left( \eta^{f(0)} \prod_{\ell \in \N} \left( \frac{ \vartheta_{\ell} }{ \eta } \right)^{f(\ell)} ; x \right) = 
\frac{k}{12} + \frac12 \sum_{\ell \in \N}  f(\ell)  {\bar B}_2(\ell x), 
$$
where $ {\bar B}_2(x) $ is 
the periodic extension of the the 
second Bernoulli polynomial, normalized in the traditional way, $B_2(x)=x^2-x+\frac16$, and 
$$
{\bar B}_2(x) = B_2( x -\ldL x \rdL )  =
\sum_{n=1}^{\infty}\, \frac{ \cos( 2n \pi x )}{(n\pi)^2}.  
$$
The following Theorem, stated only for theta blocks without theta denominator, 
suffices for our needs.  

\begin{thm} [Gritsenko, Skoruppa, Zagier]
Let $\ell,m\in\N$ and let $u,k \in \Z$.  Select 
$ {\bf{d}}=(d_1, \dots,d_{\ell})\in\N^{\ell}$.  
Define a meromorphic function 
$\TB(u;  {\bf{d}}):\SiegelH_1\times\C\to\C$ by 
$$
\TB(u; {\bf{d}})(\tau,z)=
\eta(\tau)^u\,
\prod_{i=1}^{\ell}\, \vartheta(\tau, d_iz).  
$$
We have $\TB(u; {\bf{d}})\in J_{k,m}^{ \text{\rm cusp}}$ (respectively  $ J_{k,m}$) if and only if
\begin{itemize}
\item $2k=\ell+u$,  
\item $2m= \sum_{i=1}^{\ell} d_i^2$,
\item $u+3\ell \equiv 0 \mod 24$,    
\item   $\frac{k}{12} + \frac12 \sum_{i=1}^{\ell} {\bar B}_2(d_ix)$ 
has a positive (respectively nonnegative) minimum on $[0,1]$.  
\end{itemize}
\end{thm}

For such a theta block,  we have the infinite product 
\begin{equation}\label{TBformula}
\begin{aligned}
\TB(u;d_1,\ldots,d_\ell)
=
q^{v} 
&
\prod_{j \in \N} (1-q^j)^{2k} 
\prod_{i=1}^\ell\left( \zeta^{\frac12d_i}- \zeta^{-\frac12d_i} \right) \\
&
\cdot\prod_{i=1}^\ell \prod_{j \in \N}
(1 - q^j\zeta^{d_i})(1-q^j\zeta^{-d_i})
\end{aligned}
\end{equation}
where
$v=\tfrac{u+3\ell}{24}$.

\section{Proof of Main Theorem}  

Beginning with a theta block $\phi \in J_{k,t}\wh$, 
$$
\phi(\tau,z)= 
\TB(u; {d_1,\ldots,d_\ell})(\tau,z)=
\eta(\tau)^u\,
\prod_{i=1}^{\ell}\, \vartheta(\tau, d_iz), 
$$
with order of vanishing in~$q$ equal to~$v$, 
we wish to investigate 
$\psi = (-1)^{v} \frac{ \phi | V_2 }{ \phi }$.  
It is clear that $\psi$ transforms like a Jacobi form of weight~zero 
with the same index as $\phi$; 
we will show that $\psi$ is weakly holomorphic, not just meromorphic.  
For prime $\ell$ we have
$$
c( n,r; \phi|V_\ell) =
c( \ell n, r; \phi) + \ell^{k-1} c( n/\ell, r/\ell; \phi),
$$
and for $\ell=2$, 
we have 
$$(\phi|V_2)(\tau,z)
= 2^{k-1}  \phi(2\tau,2z)+
\frac 12
\left(\phi(\tfrac12\tau,z)+\phi(\tfrac12\tau+\tfrac12,z) \right).
$$
The following equation for $\psi = (-1)^{v} \frac{ \phi | V_2 }{ \phi }$ 
will often be cited.    

\begin{prop}\label{psiprop}  
Let $u,v, k \in \Z$ and $\ell, d_1, \dots, d_{\ell} \in \N$.  
Take a theta block 
$\phi=\TB(u;d_1,\ldots,d_\ell)$ of order
 $v=\tfrac{u+3\ell}{24}$ and weight $k=\frac{\ell+u}{2}$.  
For $\psi=(-1)^v \frac{\phi|V_2}\phi$, we have 
\begin{align*}
&\psi(\tau,z)=  
(-1)^v
2^{k-1}\frac{\phi(2\tau,2z)}{\phi(\tau,z)}
+(-1)^v
\frac 12
\left( \frac{\phi(\tfrac12\tau,z)}{\phi(\tau,z)}+ \frac{\phi(\tfrac12\tau+\tfrac12,z)}{\phi(\tau,z)} \right)   =   \\
&(-1)^v 2^{k-1}q^{v} 
\prod_{j \in \N} (1{+}q^j)^{2k} 
\prod_{i=1}^\ell\left( \zeta^{\frac12d_i}{+} \zeta^{-\frac12d_i} \right)    
\prod_{i=1}^\ell \prod_{j \in \N}
(1 {+} q^j\zeta^{d_i})(1+q^j\zeta^{-d_i})
\\
&
+
\tfrac12q^{-\frac12v}\left(
(-1)^v
\prod_{2\notdivides j} (1-q^{\frac j 2})^{2k} 
\prod_{i=1}^\ell \prod_{2\notdivides j}
(1 -q^{\frac j 2}\zeta^{d_i})
(1-q^{\frac j 2}\zeta^{-d_i})\right.
\\&
\qquad\quad\left.
+\prod_{2\notdivides j} (1+ q^{\frac j 2})^{2k} 
\prod_{i=1}^\ell \prod_{2\notdivides j}
(1 +  q^{\frac j 2}\zeta^{d_i})
(1+q^{\frac j 2}\zeta^{-d_i})
\right) .
\end{align*}
\end{prop}

\begin{proof}
This formula follows from Equation (\ref{TBformula})
and  algebra.
\end{proof}

\begin{cor}
\label{psiwh} 
Let $u,v \in \Z$ and $k,m,\ell, d_1, \dots, d_{\ell} \in \N$.  
Let $\phi=\TB(u;d_1,\ldots,d_\ell)$ be a theta block of order
 $v=\tfrac{u+3\ell}{24}$ and index~$m$.  
Then $\psi=(-1)^v \frac{\phi|V_2}\phi \in J_{0,m}\wh$ 
and $\psi$ has integral Fourier coefficients.  
\end{cor}

Now we examine the support of~$\psi$.  
The first term in the equation of Proposition~\ref{psiprop} involves 
$\frac{\phi(2\tau,2z)}{\phi(\tau,z)}$ and we consider it separately.

\begin{lm}\label{lm3A}
Let $\phi=\TB(u;d_1,\ldots,d_\ell)$ be a theta block.
Then
$$\vJ(\frac{\phi(2\tau,2z)}{\phi(\tau,z)})
=\vJ(\phi).
$$
\end{lm}

\begin{proof}
First we note the equality of the Jacobi valuations of the {\sl atoms:\/}    
$\vJ(1-q^j\zeta^d)=\vJ(1+q^j\zeta^d)$.  
We use Proposition~\ref{psiprop} to write 
\begin{multline*}
\vJ\left(
\prod_{j \in \N} 
\left(
(1+q^j)^{2k}  
\prod_{i=1}^\ell
(1 + q^j\zeta^{d_i})(1+q^j\zeta^{-d_i})
\right)
\right)\\
\begin{aligned}
&=
\bigcup_{m=1}^\infty
\vJ\left(
\prod_{j=1}^m 
\left(
(1+q^j)^{2k}  
\prod_{i=1}^\ell
(1 + q^j\zeta^{d_i})(1+q^j\zeta^{-d_i})
\right)\right)  
\\
&=
\bigcup_{m=1}^\infty
\vJ\left(
\prod_{j=1}^m 
\left(
(1-q^j)^{2k}  
\prod_{i=1}^\ell
(1 -q^j\zeta^{d_i})(1-q^j\zeta^{-d_i})
\right)\right)
\\
&=
\vJ\left(
\prod_{j \in \N} 
\left(
(1-q^j)^{2k}  
\prod_{i=1}^\ell
(1 - q^j\zeta^{d_i})(1-q^j\zeta^{-d_i})
\right)\right) .  
\end{aligned}
\end{multline*}
The first and third equalities  follow from Lemma~\ref{davelemma}.  
The second equality follows from the valuation property from Theorem~\ref{Good} 
on Laurent polynomials and 
the equality of the valuations of the atoms. 
Adding $\vJ\left(q^v\prod_{i=1}^\ell ( \zeta^{\frac12d_i}+ \zeta^{-\frac12d_i}  ) \right)=
\vJ \left(q^v\prod_{i=1}^\ell ( \zeta^{\frac12d_i}- \zeta^{-\frac12d_i} ) \right)$ to the left and right hand sides, 
in $R_J(24)$ say, 
and applying Theorem \ref{Good} completes the proof.
\end{proof}

We use some combinatorial lemmas,  
set $ \LA=\{ -\ell, \dots,-1,1,\dots,\ell \}$.   

\begin{lm}\label{combLemma1}
Let $d_1,\ldots,d_\ell\in\N$.  
Set $e_i = \sgn(i) d_{ | i| }$ for  
$i \in \LA$.
Fix $a\in\N$.
Then we have
$$\sum_{m_1,\ldots,m_a \in \LA:\, -\ell \le m_1<\ldots<m_a \le \ell}
(e_{m_1}+\cdots+e_{m_a})^2
=
\binomial{2\ell-2}{a-1}
\sum_{i \in \LA}
e_{i}^2.
$$
\end{lm}

\begin{proof}
When the lefthand side above  
is expanded, we have
\begin{itemize}
\item[(1)]
For each $1\le i \le \ell$,
$e_i^2$ and $e_{-i}^2$  occur
$\binomial {2\ell-1}{a-1}$ times.
\item[(2)]
For each $1\le i \le \ell$,
$e_i e_{-i}$ and $e_{-i}e_i $
each occur $\binomial {2\ell-2}{a-2}$ times.  
\item[(3)]
For
$1\le i,j \le \ell$ with $i\ne j$,
we have
$e_ie_{j}$,
$e_{-i}e_{-j}$,
$e_ie_{-j}$,
and $e_{-i}e_j$
each occurs $2\binomial{2\ell-2}{a-2}$ times.
\end{itemize}
Adding up each category,
since the cases within item (3) cancel themselves out,
we get that
\begin{align*}
\sum_{m_1,\ldots,m_a \in \LA: \, -\ell \le m_1<\dots<m_a \le \ell}
&(e_{m_1} + \cdots + e_{m_a})^2  \\
&=
\sum_{i=1}^{\ell}
(
2\binomial{2\ell-1}{a-1} - 2\binomial{2\ell-2}{a-2}
)\sum_{i=1}^{\ell}
e_{i}^2.
\end{align*}
Then the result follows from  
$\binomial{2\ell-1}{a-1}-\binomial{2\ell-2}{a-2}
=\binomial{2\ell-2}{a-1}.$
\end{proof}

\begin{lm}\label{combLemma2}
Let $d_1,\ldots,d_\ell\in\N$.
Set $e_i = \sgn(i) d_{ | i| }$ for  $i \in \LA$.
Fix $b_1,\ldots,b_\beta\in\N$.
For any subset $S\subseteq\LA$,
denote
$e_S=\sum_{i\in S}e_i.$
Then
$$\sum_{S_1,\ldots,S_\beta}
(\sum_{i=1}^{\beta} e_{S_i})^2
\text{
is an integral multiple of }
\sum_{i \in \LA}
e_{i}^2, $$
where the outer sum on the lefthand side
is over all $S_1,\ldots,S_\beta$,  where each
$S_i\subseteq\LA$ with $|S_i|=b_i$.
\end{lm}

\begin{proof}
We proceed by induction on $\beta$.
The case $\beta=1$ follows from Lemma \ref{combLemma1}.
So assume $\beta>1$.  
The key is that 
$$
\sum_{S_1,\ldots,S_\beta}
(\sum_{i=1}^{\beta} e_{S_i})^2
=\sum_{S_1,\ldots,S_\beta}
(e_{S_1}+\cdots+e_{S_{\beta-1}}-
e_{S_\beta})^2.
$$
The reason for this is that $S \mapsto -S$ gives a involution on 
subsets $S \subseteq \LA $ of any fixed size and 
satisfies $e_{-S}=-e_{S}$.  
Then adding the above righthand side to the lefthand side yields
$$
2\sum_{S_1,\ldots,S_\beta}
(\sum_{i=1}^{\beta} e_{S_i})^2
=\sum_{S_1,\ldots,S_\beta}
\left(
2
(e_{S_1}+\cdots+e_{S_{\beta-1}})^2
+
2e_{S_\beta}^2)
\right.
$$
The result follows by induction.   
\end{proof}

\begin{thm}\label{borcherdsproductseverywhere}
{\rm (Borcherds Products Everywhere)\/}  
Fix $\ell \in \N$ and $u \in \Z$ with $\ell+u$ even.  
Let $d_1,\ldots,d_\ell \in \N$ with $d_1+ \cdots + d_\ell $ even.  
Assume that $v = \frac1{24}(u+3\ell) \in \N$.  
If we set $k=\frac12(\ell+u)$ 
and $t=\frac12( d_1^2+ \cdots +d_{\ell}^2)$ 
then we have 
$\phi=\eta^u\,\prod_{i=1}^{\ell}\, \vartheta_{d_i}\in J_{k,t}\mero$.  
If $v$ is odd, additionally assume that $\phi\in J_{k,t}$.  
For $\psi=(-1)^v\frac{\phi|V_2}\phi$,  
we have the following:  
\begin{itemize}
\item[(1)]
$\psi \in J_{0,t}\wh$ and $c(n,r;\psi)\ge0$
for all $(n,r)$ with $4tn-r^2\le 0$.  
\item[(2)] 
There is a $k'\in\N$ such that 
$\Borch(\psi)\in M_{k'}(K(t))$ is a holomorphic Borcherds product with trivial character.  
\item[(3)]
$\Borch(\psi)$ is antisymmetric when $v$ is an odd power of two and otherwise symmetric.  
\item[(4)]
If $v=1$, then
$\Borch(\psi)\in M_k(K(t))$
and
$\Borch(\psi)$ and $Grit(\phi)$
have the same first and second Fourier Jacobi coefficients.
\end{itemize}
\end{thm}

\begin{proof}
That $\psi$ is weakly holomorphic and has integral Fourier coefficients was proven in Corollary~\ref{psiwh}.  
We show that the  Fourier coefficients of singular indices are nonnegative.  
Consider the formula for $\psi$ from Proposition~\ref{psiprop}.  
Consider the case when $v$ is odd.  
Since $\phi\in J_{k,t}$, we have $4tn-r^2\ge 0$ for $(n,r) \in \supp(\phi)$ and the 
same inequality hold on the convex hull $\vJ(\phi)=\conv(\supp(\phi); \rayA)$.  
By Lemma \ref{lm3A},  
$\supp(\frac{\phi(2\tau,2z)}{\phi(\tau,z)}) \subseteq \vJ(\phi)$ also 
does not contain any $(n,r)$ with $4tn-r^2 < 0$.  
When $v$ is even, all the coefficients of $\phi(2\tau,2z)/\phi(\tau,z)$ 
are nonnegative anyhow.  
The remaining terms in this equation are 
\begin{equation*} 
\begin{aligned}
\frac{(-1)^v}{2} &\left(
\frac{\phi(\tfrac12\tau,z)}{\phi(\tau,z)}+
\frac{\phi(\tfrac12\tau+\tfrac12,z)}{\phi(\tau,z)}
\right)
\!\!\!\!\!\!\!\!\!\!\!\!\!\!\!\!\!\!\!\!\!\!\!\!
\!\!\!\!\!\!\!\!\!\!\!\!\!\!\!\!\!\!\!\!\!\!\!\!\!\!\!\!
\\
=
&
\frac12 
q^{-\frac12v}\left(
(-1)^v
\prod_{2\notdivides j} (1-q^{\frac j 2})^{2k} 
\prod_{i=1}^\ell \prod_{2\notdivides j}
(1 -q^{\frac j 2}\zeta^{d_i})
(1-q^{\frac j 2}\zeta^{-d_i})\right.
\\&
\qquad\quad\left.
+\prod_{2\notdivides j} (1+ q^{\frac j 2})^{2k} 
\prod_{i=1}^\ell \prod_{2\notdivides j}
(1 +  q^{\frac j 2}\zeta^{d_i})
(1+q^{\frac j 2}\zeta^{-d_i})
\right) .
\end{aligned}
\end{equation*}
It is clear that the nonzero coefficients are all positive 
because they are coefficients of a formal series of the form $f(q^{1/2}, \zeta) \pm f(-q^{1/2}, \zeta) $,  
where~$f$ has all nonnegative coefficients.  
The only terms of the singular part of $\psi$ that might have negative coefficients come from the 
first term where a monomial $q^n\zeta^r$ with $4tn-r^2=0$ might be supported.  
Therefore, if the multiplicity $\sum_{\lambda \in \N} c(\lambda^2 n_o m_o, \lambda r_o)$ 
of $\Borch(\psi)$ on $\Hum\mymat{n_o}{r_o/2}{r_o/2}{tm_o}$ has a negative summand, 
$c(\lambda_1^2 n_o m_o, \lambda_1 r_o)<0$, for some $\lambda_1 \in \N$, 
then $4t( \lambda_1^2 n_o m_o ) - ( \lambda_1 r_o)^2 =0$ and 
$\Hum\mymat{n_o}{r_o/2}{r_o/2}{tm_o}$ is empty.  
Thus $\Borch(\psi)$ is a holomorphic Borcherds product by Theorem~\ref{BP}.

To prove that the character is actually trivial, 
we use the notation 
\begin{align*}
A&=\tfrac{1}{24}\sum_{r\in\Z}c(0,r;\psi), \quad 
C=\tfrac{1}{4}\sum_{r\in\Z} r^2 c(0,r;\psi),
\\
D_1&=\sum_{n<0,r\in\Z}\sigma_1(-n)\,c(n,r;\psi).
\end{align*}

\def\SS{{\cal S}}
First, we prove $t|C$.
From Proposition~\ref{psiprop}, noting that $v>0$ by assumption here, 
we have that
$
c(0,r; \psi)
$ is the coefficient of 
$q^{\frac12 v}\zeta^r$
of
$$
\prod_{2\notdivides j} (1+ q^{\frac j 2})^{2k} 
\prod_{i=1}^\ell \prod_{2\notdivides j}
(1 +  q^{\frac j 2}\zeta^{d_i})
(1+q^{\frac j 2}\zeta^{-d_i})
.$$
Let $e_i=\sgn(i) d_{ | i | }$ 
for $i \in \LA =\{ -\ell,\dots,-1,1,\dots,\ell\}$.
Then
$
c(0,r;\psi)
$ is the coefficient of 
$q^{\frac12 v}\zeta^r$
in 
$$
\prod_{2\notdivides j} (1+ q^{\frac j 2})^{2k} 
\prod_{i \in \LA} \prod_{2\notdivides j}
(1 +  q^{\frac j 2}\zeta^{e_i})
.$$
Let $a_j$ be the number of factors of $q^{j/2}$ selected from the product, and 
$b_j$ be the total number of 
$q^{j/2}\zeta^{e_{-\ell}}, \dots, q^{j/2}\zeta^{e_{-1}}, q^{j/2}\zeta^{e_{1}}, \dots,q^{j/2}\zeta^{e_{ \ell}}$ selected from the product.  
Denote the set
\begin{equation*}
\SS_v
 = 
\{
(a_1,a_3,a_5,\ldots;
b_1,b_3,\ldots:
a_i,b_i\in\N\cup\{0\},
\,
\sum_{i=1}(a_i + b_i)i
=v
\}.
\end{equation*}
For $S\subset\LA$, denote $e_S=\sum_{i\in S}e_i$.
Then $4C$ is the sum over each of the above elements of $\SS_v$ of the sum
\begin{equation}\label{eq3}
\left(\binomial{2k}{a_1}\binomial{2k}{a_3}\cdots\right)
\sum_{(S_1,S_3,\ldots):|S_i|=b_i} \left(
\sum_{i}
e_{S_i}
\right)^2
\end{equation}
where each $S_i\subseteq \LA$.
By Lemma \ref{combLemma2},
the above
Equation (\ref{eq3})
is an integral multiple
of
$
\sum_{i \in \LA}
e_{i}^2$
This proves that
$4C$ is a multiple of $\sum_{i \in \LA}
e_{i}^2$.
Since $4t=\sum_{i \in \LA}
e_{i}^2$, this proves $t|C$.  
In particular, $C$ is integral and this shows that $2B \equiv 2C \equiv 0 \mod 2$.  

It is clear that $D_1\in\Z$.
By Proposition~\ref{Identity},  we have 
$tA-tD_1-C=0$,  and 
we deduce that $A\in\Z$.
This says the Borcherds product is a paramodular form with a trivial character on $K(t)$.

\def\zeroIndicator{Z}

Finally, for item (4), When $v=1$, we can say more.
From the formula for $\psi$,
we see
$$\psi = 
2k + \sum_{i=1}^\ell (\zeta^{d_i}+\zeta^{-d_i})+
q(\cdots)+q^2(\cdots)+\cdots.
$$
so that
\begin{align*}
A&=\tfrac{1}{24}\sum_{r\in\Z}c(0,r;\psi) = \tfrac1{24}(2k+2\ell)=1,
\\
C&=\tfrac{1}{4}\sum_{r\in\Z} r^2 c(0,r;\psi)
=\tfrac{1}{4}\sum_{i=1}^\ell 2 d_i^2 = \tfrac{1}{4} 2(2t)=t,
\\
D_1&=\sum_{n<0,r\in\Z}\sigma_1(-n)\,c(n,r;\psi)=0.
\end{align*}
The equation  $c(0,0;\psi)=2k$ says the 
weight of $\Borch(\psi)$ is $k$.  
By Theorem~\ref{BP} the first Fourier Jacobi coefficient of $\Borch(\psi)$ is 
$$
\eta^{c(0,0;\psi)} \prod_{r \in \N} \left( \frac{\vartheta_{r}}{\eta} \right)^{c(0,r;\psi)}
=
\eta^{2k} \prod_{i=1}^{\ell}  \left( \frac{\vartheta_{d_i}}{\eta} \right),  
$$
which is exactly $\phi$.  
In view of the formula
$$
\Borch(\psi)={\tilde \phi} \exp \left(-\Grit(\psi) \right) 
= \phi \xi^C 
 \left( 1 - \psi \xi^t + \frac{1}{2!} \psi|V_2 \xi^{2t} + \cdots  \right), 
$$
the second Fourier Jacobi coefficient of $\Borch(\psi)$
is $\phi(-\psi)=\phi|V_2$, which is the second Fourier Jacobi coefficient of
$\Grit(\phi)$. 
This completes the proof, except for item (3), which we postpone until the next section.  
\end{proof}

\section{Proof of Symmetry and Antisymmetry}  

This section is devoted to the proof that a paramodular Borcherds Product,  
constructed as in Theorem~\ref{borcherdsproductseverywhere}, is antisymmetric if and only if 
the vanishing order of the theta block is an odd power of two.  
One can glimpse this in Table~$1$, 
where odd weights occur only at vanishing orders~$2$ and~$8$.  
In Table~$1$, 
for the special case of level one, antisymmetric and odd weight are equivalent.  
The following theorem gives the general result on the parity of $D_0$ in terms of the hypotheses 
of Theorem~~\ref{borcherdsproductseverywhere} and thereby completes the proof of item~{(3)} 
in that theorem.  

\begin{thm}\label{antisym}
Let $k, u \in \Z$ and $v, \ell, t, d_1,\ldots,d_\ell \in \N$.  
Assume that 
$2k=\ell+u$, $2t=d_1^2+\dots +d_{\ell}^2$ and $24v=u+3\ell$, 
so that 
$\phi=\eta^u\,
\prod_{i=1}^{\ell}\, \vartheta_{d_i}\in J_{k,t}\mero$.  
For $\psi=(-1)^v\frac{\phi|V_2}\phi$,  
let $D_0 = \sum_{n \in \Z: n< 0} 
\sigma_0(-n) c(n,0;\psi)$.  
We have 
$$D_0 \equiv 1 \mod 2 
\iff
\exists \text{ odd } \beta \in \N: v=2^\beta.
$$
Equivalently, $\Borch(\psi)$ is antisymmetric if and only if $v$ is an odd power of two.  
\end{thm}
\smallskip

If $n <0$ then only the second term in Proposition~\ref{psiprop} 
can contribute to $c(n,0;\psi)$.  Thus $c(n,0;\psi)$ is the coefficient of 
$q^{n+\frac12 v} \zeta^0$ in

$$
\prod_{2\notdivides j} (1+ q^{\frac j 2})^{2k} 
\prod_{i=1}^\ell \prod_{2\notdivides j}
(1 +  q^{\frac j 2}\zeta^{d_i})
(1+q^{\frac j 2}\zeta^{-d_i})
.$$
Set $e_i = \sgn(i) d_{ | i| }$ for 
$i \in \LA =\{ -\ell,\dots,-1,1,\dots,\ell\}$.
Then
$
c(n,0 ;\psi)
$ is the coefficient of 
$q^{n+\frac12 v}\zeta^0$
in 
\begin{equation}
\label{dot}
\prod_{  \text{ odd } j \in \N }\left((1+ q^{\frac j 2})^{2k}   \prod_{i \in \LA} (1 +  q^{\frac j 2}\zeta^{e_i})  \right).  
\end{equation}
We multiply out the infinite product~{(\ref{dot})} 
to an infinite sum.  For odd~$j$, let 
$a_j$ be the number of factors of $q^{j/2}$ selected from the product, and 
$b_j$ be the total number of 
$q^{j/2}\zeta^{e_{-\ell}}, \dots, q^{j/2}\zeta^{e_{-1}}, q^{j/2}\zeta^{e_{1}}, \dots,q^{j/2}\zeta^{e_{ \ell}}$ 
selected from the product.  
Knowing $a_j$ determines the contribution of the $q^{j/2}$ factors but, for fixed $b_j$, 
we still need to sum over all subsets $S_j \subseteq \LA$ with $ | S_j| = b_j$ in order to get the 
exponent of $\zeta$.  
\begin{df}
For $h \in \Z$, define the set $\TT_h=\TT(h)$ by 
$$
\TT_h
 = 
\{
(a_1,a_3,a_5,\ldots) \in \prod_{\text{\rm odd $i \in \N$ }} ^{\infty} (\N \cup \{0\}) : \sum_{\text{\rm odd $i \in \N$  }} i a_i =h
\}.
$$
\end{df}
For $S\subset\LA$, denote $e_S=\sum_{i\in S}e_i$.
Then multiplying out~{(\ref{dot})} gives
\begin{align*}
&\sum_{\substack{A,B \in \N \cup \{0\} }} \, 
\sum_{\substack{ (a_1,a_3,\dots) \in \TT_A, \\ (b_1,b_3,\dots) \in \TT_B }}
\left(  \binomial{2k}{a_1}\binomial{2k}{a_3}\cdots \right) 
q^{\frac{A+B}{2}}
\sum_{\substack{ S_1,S_3,\ldots \subseteq \LA: \\ | S_i| = b_i }} \,
 \prod_{\text{\rm odd $i \in \N$  }} 
\zeta^{e_{S_i}} .  
\end{align*}
In order to grab the coefficients of the monomials with $\zeta^0$, define
$$\zeroIndicator(x)=\begin{cases}
1 & \text{if } x=0,\\
0 & \text{otherwise}.
\end{cases}
$$
For $n<0$, we get that 
$c(n,0;\psi)$ is the total sum of
\begin{equation}\label{eq5}
\sum_{\substack{A,B \in \N \cup \{0\}: \\ A+B=v+2n }} \, 
\sum_{\substack{ (a_1,a_3,\dots) \in \TT_A, \\ (b_1,b_3,\dots) \in \TT_B }}
\left(  \binomial{2k}{a_1}\binomial{2k}{a_3}\cdots \right) 
\sum_{\substack{ S_1,S_3,\ldots \subseteq \LA: \\ | S_i| = b_i }} \,
\zeroIndicator\left(  \sum_{\text{\rm odd $i \in \N$  }}   e_{S_i}  \right).
\end{equation}
A series of simplifications will show that the parity of  $D_0 $ only depends upon $v$.  
For $S \subseteq \LA$, we have $e_{-S}=-e_{S}$ and if, in the final summation, 
we pair $(S_1,S_3,\dots)$ with $(-S_1,-S_3,\dots)$ when these are distinct, 
then 
the sum for $c(n,0;\psi)$ modulo two 
may be restricted to sum only over $S_i$ with $S_i=-S_i$.  
Hence $c(n,0;\psi)$ is congruent to
\begin{equation*}
\sum_{\substack{A,B \in \N \cup \{0\}: \\ A+B=v+2n }} \, 
\sum_{\substack{ (a_1,a_3,\dots) \in \TT_A, \\ (b_1,b_3,\dots) \in \TT_B }}
\left(  \binomial{2k}{a_1}\binomial{2k}{a_3}\cdots \right) 
\sum_{\substack{ S_1,S_3,\ldots \subseteq \LA: \\ | S_i| = b_i \\ \text{ and } -S_i=S_i}} \,
\zeroIndicator\left( \sum_{\text{\rm odd $i \in \N$  }}  e_{S_i}  \right).
\end{equation*}
Restricting to subsets $S_i$ with $S_i = -S_i$ is the same as summing over 
$S_i$ of the form $T \cup (-T)$, where $T \subseteq \{1,2,\dots,\ell\}$; 
and in each such case we have $e_{S_i}=e_{ T \cup (-T) } =0$.  Thus we have
$$
\sum_{\substack{ S_1,S_3,\ldots \subseteq \LA: \\ | S_i| = b_i \\ \text{ and } -S_i=S_i}} \,
\zeroIndicator\left(  \sum_{\text{\rm odd $i \in \N$  }}  e_{S_i}  \right)
=  
\sum_{\substack{ T_1,T_3,\ldots \subseteq \{1,2,\dots,\ell\}: \\ | T_i| = \frac12 b_i }} \,
1 
=\prod_{\text{\rm odd $i \in \N$  }} 
\binomial{\ell}{\frac12 b_i} 
$$
and this is zero if any $b_i$ is odd.  
Thus we may restrict the summation for $c(n,0;\psi)$ to those $(b_1,b_3,\dots) \in \TT_B$ 
where every $b_i$ is even.  The same is actually true for the $a_i$ because 
$\binomial{2k}{a_i} \equiv 0 \mod 2$ when $a_i$ is odd.  The proof is that,  by definition, 
$\binomial{2k}{a}= \frac{2k}{a} \binomial{2k-1}{a-1}$ and if $a$ is odd then 
$\binomial{2k}{a_i} \equiv 2k  \binomial{2k-1}{a-1} \equiv 0 \mod 2$.  
Furthermore, for even $a$, we have $\binomial{2k}{a} \equiv \binomial{k}{ a/2} \mod 2$.  
The proof is to ignore all the odd factors in 
$$
\binom{2m}{2j}= \frac{(2m)(2m-1)2(m-1)(2m-3)2(m-2) \cdots}{(2j)(2j-1)2(j-1)(2j-3)2(j-2) \cdots}, 
$$
and to cancel a $2$ for each even factor, leaving $\binomial{m}{j}$.  
Let ${\bar A} = \frac12 A$, ${\bar B} = \frac12 B$, ${\bar a_i} = \frac12 a_i$ 
and ${\bar b_i} = \frac12 b_i$.  For negative~$n$, we have
$$
c(n,0;\psi) \equiv 
\sum_{\substack{{\bar A}, {\bar B} \in \N \cup \{0\}: \\ {\bar A}+ {\bar B}= \frac12 v+n }} \, 
\sum_{\substack{ ({\bar a_1},{\bar a_3},\dots) \in \TT_{\bar A}, \\ ({\bar b_1},{\bar b_3},\dots) \in \TT_{\bar B} }}
\left(  \binomial{k}{\bar a_1}\binomial{k}{\bar a_3}\cdots \right) 
\left(  \binomial{\ell}{\bar b_1}\binomial{\ell}{\bar b_3}\cdots \right) . 
$$
We reorganize this summation by setting 
$r_i = {\bar a_i} + {\bar b_i}$ for odd~$i$ 
so that $(r_1,r_3,\dots) \in \TT({{\bar A}+{\bar B}})=\TT({\frac12 v +n})$.  
For negative~$n$, we have
$$
c(n,0;\psi) \equiv 
\sum_{\substack{ ({ r_1},{ r_3},\dots) \in \TT({ \frac12 v+n}) }}\,
\prod_{ \text{odd } i \in \N} 
\left( \sum_{ {\bar a_i} + {\bar b_i}=r_i} \binom{k}{\bar a_i} \binom{\ell}{\bar b_i} \right) . 
$$
Now we make use of the binomial convolution identity, valid for 
$n_1$, $n_2$, $k \in \N \cup \{0\}$,
$$
\sum_{j=0}^k \binom{n_1}{k-j} \binom{n_2}{j} = \binom{n_1+n_2}{k} .
$$
The proof is to count the number of size~$k$ subsets of $n_1+n_2$~items by 
breaking the count into cases according to the number of elements in the subset that are from the 
first $n_1$ items and the number that are from the second $n_2$ items.  
Thus, 
$$
c(n,0;\psi) \equiv 
\sum_{\substack{ ({ r_1},{ r_3},\dots) \in \TT({ \frac12 v+n}) }}\,
\prod_{ \text{odd } i \in \N} 
\binom{k+\ell}{r_i} 
$$
and remembering the theta block satisfies $k+\ell=12v$,  we have 
$$
c(n,0;\psi) \equiv
\sum_{\substack{ ({ r_1},{ r_3},\dots) \in \TT({ \frac12 v+n}) }}\,
\prod_{ \text{odd } i \in \N} 
\binom{12v}{r_i} .  
$$
This shows that the parity of $D_0$ depends only on $v$ and concludes 
the first series of reductions.  At this point, one might finish by computing any 
example for each~$v$.  It is just as easy, however, 
to continue with the formula at hand, which amounts at least formally to a specific example.  
According to the definition $D_0 = \sum_{n \in \Z: n< 0} \sigma_0(-n) c(n,0;\psi)$, 
we need only consider $n$ such that $\sigma_0(-n) \not\equiv 0 \mod 2$; this condition 
holds if and only if $-n$ is a square.  So we consider, for $n >0$, 
$$
c(-n^2,0;\psi) \equiv
\sum_{\substack{ ({ r_1},{ r_3},\dots) \in \TT({ \frac12 v-n^2}) }}\,
\prod_{ \text{odd } i \in \N} 
\binom{12v}{r_i} .  
$$
We will show that $c(-n^2,0;\psi) \not\equiv 0 \mod 2$ implies that 
$n= 2^{\frac{\beta-1}{2}} m$ for some odd $\beta,m\in \N$.  
Let $v= 2^{\beta} w$, where $w \in \N$ is odd. For $n>0$, 
$$
c(-n^2,0;\psi) \equiv
\sum_{\substack{ ({ r_1},{ r_3},\dots) \in \TT({ \frac12 v-n^2}) }}\,
\prod_{ \text{odd } i \in \N} 
\binom{2^{\beta+2}\,3w}{r_i} .  
$$
As before, if $\binomial{2^{\beta+2}\,3w}{r_i} $ is odd then $2^{\beta+2} | r_i$, 
so we may restrict this sum to $r_i$ that are divisible by   $2^{\beta+2}$.  
Let ${\bar r_i} = r_i/2^{\beta+2}$, so that
$$
c(-n^2,0;\psi) \equiv
\sum_{\substack{ ({\bar r_1},{\bar  r_3},\dots) \in \TT({     \frac{ \frac12 v-n^2}{2^{\beta+2}})        } }}\,
\prod_{ \text{odd } i \in \N} 
\binom{3w}{\bar r_i} .  
$$
If this summation is not empty, then we have $2^{\beta+2} | (\frac12 v -n^2)$ 
or, equivalently,   $2^{\beta+2} | (2^{\beta-1}w -n^2)$.  
This implies  $2^{\beta-1} | n^2$ and  $8 | ( w - \frac{n^2}{2^{\beta-1}})$.  
Since $w$ is odd, the integer $ \frac{n^2}{2^{\beta-1}}$ is odd, 
which implies that $\beta-1$ is even and that $n= 2^{\frac{\beta-1}{2}} m$ for some odd $m \in \N$, 
as claimed.  Thus we have 
$$
c(-n^2,0;\psi) \equiv
\sum_{\substack{ ({\bar r_1},{\bar  r_3},\dots) \in \TT({     \frac{ w-m^2}{8}        }) }}\,
\prod_{ \text{odd } i \in \N} 
\binom{3w}{\bar r_i} , 
$$
since $   \frac{ \frac12 v-n^2}{2^{\beta+2}}  =   \frac{ 2^{\beta-1} w-2^{\beta-1}m^2}{2^{\beta+2}}  =  \frac{ w-m^2}{8} $.  
If $c(-n^2,0;\psi) \not\equiv 0 \mod 2$ then $ w \equiv m^2 \mod 8$, 
which implies $w \equiv 1 \mod 8$.  In this case, let 
$w= 1 + 8 \mu$ for $ \mu \in \N \cup \{0\}$ and $m = 2 \lambda +1$ for $ \lambda \in \N \cup \{0\}$.  
In terms of $\mu$ and $\lambda$, we have 
$\frac{w -m^2}{8}= \frac{ 1+8 \mu - (2 \lambda+1)^2}{8} = \mu - \binomial {\lambda +1}{2}$.  
So 
$$
c(-n^2,0;\psi) \equiv
\sum_{\substack{ ({\bar r_1},{\bar  r_3},\dots) \in \TT({    \mu - \binomial {\lambda +1}{2}      }) }}\,
\prod_{ \text{odd } i \in \N} 
\binom{3w}{\bar r_i} .  
$$
Since all nonzero $c(-n^2,0;\psi)$ are of this form, we calculate $D_0$ as
\begin{align*}
D_0 &\equiv  
\sum_{n \in \Z: n< 0} \sigma_0(-n) c(n,0;\psi)  \equiv 
\sum_{n \in \N} \sigma_0(n^2) c(-n^2,0;\psi)  \\
&\equiv \sum_{n \in \N}  c(-n^2,0;\psi) \equiv 
\sum_{\lambda \ge 0} \,
\sum_{\substack{ ({\bar r_1},{\bar  r_3},\dots) \in \TT({    \mu - \binomial {\lambda +1}{2}      }) }}\,
\prod_{ \text{odd } i \in \N} 
\binom{3w}{\bar r_i} .  
\end{align*}

{\it Claim:   
For $v = 2^{\beta} (1+ 8 \mu)$, 
$D_0 \equiv 0 \mod 2$ if and only if $ \mu \ge 1$.  \/}  
\smallskip

To prove the easy direction of this claim, 
note that $\mu=0$ forces $\lambda=0$ and all ${\bar r_i} =0$, 
so that $D_0 \equiv 1 \mod 2$.  
For the harder direction, fix $\mu \ge 1$ and $w= 1 + 8 \mu$.  
Note that $3w=3+24 \mu$.  
\begin{df}
For $n,\mu \in \Z$ with $n \ge 0$ and $\mu \ge 1$, define
$$
H(n) = 
\sum_{\substack{ ({ r_1},{ r_3},\dots) \in \TT_{   n     } }}\,
\prod_{ \text{\rm odd } i \in \N} 
\binom{3+24 \mu}{ r_i} .  
$$
\end{df}
Note that we have 
$D_0 \equiv  \sum_{\lambda \ge 0} \, H \left(\mu - \binomial{\lambda+1}{2} \right)$.  
Consider the generating function for $H(n)$, 
$$
\sum_{n \ge 0} H(n) q^n = 
(1+q)^{3+24\mu} 
(1+q^3)^{3+24\mu} \cdots = 
\prod_{ \text{\rm odd } j \in \N } (1+q^j)^{3+24\mu} .  
$$
We finally write $D_0$ in terms of modular forms modulo two:
$$
D_0 \equiv 
\Coeff\left( q^{\mu},   ( \sum_{\lambda \ge 0}  q^\binomial{\lambda+1}{2} ) ( \sum_{n \ge 0} H(n) q^n )     \right) 
\mod 2.  
$$
\begin{lm}
$$
\prod_{\text{\rm odd }  j \in \N}  
(1+q^j) 
\equiv 
\dfrac{1}{  \prod_{ j \in \N} (1-q^j)          } 
\mod 2.   
$$
\end{lm}
\begin{proof}
\begin{align*}
\prod_{\text{\rm odd }  j \in \N}  (1+q^j) 
&\equiv 
\dfrac{ \prod_{\text{\rm odd }  j \in \N}  (1-q^j) \,   \prod_{\text{\rm even }  j \in \N}  (1-q^j)   }{ \prod_{\text{\rm even }  j \in \N}  (1-q^j)   }
=
\dfrac{ \prod_{  j \in \N}  (1-q^j)  }{ \prod_{ j \in \N}  (1-q^{2j})   }  \\ 
&= 
\dfrac{1}{  \prod_{ j \in \N} (1+q^j)          }  
\equiv 
\dfrac{1}{  \prod_{ j \in \N} (1-q^j)          } .  
\end{align*}
\end{proof}
This Lemma shows us that 
$$
\sum_{n \ge 0} H(n) q^n 
\equiv 
\dfrac{1}{  \prod_{ j \in \N} (1-q^j)^{ 3+24\mu }    } 
\mod 2.  
$$
\begin{lm}
$$
\prod_{  j \in \N}  (1-q^j)^3 
\equiv 
\sum_{n=0}^{\infty}  q^{ \binom{n+1}{2} }
 \mod 2.   
$$
\end{lm}
\begin{proof}
The classical Euler identity or the Jacobi triple product identity 
gives us 
 $$
\eta(\tau)^3 = 
\sum_{n=0}^{\infty}
(-1)^n (2n+1) q^{ \frac{ (2n+1)^2 }{8} }.  
$$
Consequently we have 
$$
q^{\frac18} 
\prod_{  j \in \N}  (1-q^j)^3 
\equiv 
\sum_{n=0}^{\infty} q^{ \frac{ n(n+1)  }{2}  } q^{\frac18}  
\mod 2,
$$
which proves the lemma.  
\end{proof}
Now we may finish the derivation of the parity of $D_0$.  
\begin{align*}
D_0   &\equiv 
\Coeff\left( q^{\mu},   ( \sum_{\lambda \ge 0}  q^\binomial{\lambda+1}{2} ) ( \sum_{n \ge 0} H(n) q^n )     \right) \\
&\equiv 
\Coeff\left( q^{\mu},   (\prod_{  j \in \N}  (1-q^j)^3  ) ( \dfrac{1}{  \prod_{ j \in \N} (1-q^j)^{ 3+24\mu }    }  )     \right) \\ 
&=
\Coeff\left( q^{\mu},    \dfrac{1}{  \prod_{ j \in \N} (1-q^j)^{ 24\mu }    }      \right)  
=
\Coeff\left( q^{\mu},    \dfrac{1}{ ( \prod_{ j \in \N} (1-q^j)^{ 3 } )^{8\mu}   }      \right) .  
\end{align*}
Now 
$ ( \prod_{ j \in \N} (1-q^j)^{ 3 } )^{8\mu} \equiv  ( \sum_{n=0}^{\infty}  q^{ \binom{n+1}{2} }  )^{8\mu} $  
and squaring is a linear homomorphism modulo two so that 
$( \sum_{n=0}^{\infty}  q^{ \binom{n+1}{2} }  )^{8}  \equiv  \sum_{n=0}^{\infty}  q^{ 8\binom{n+1}{2} } $.  
Since we are in the case $\mu \ge 1$, let $\mu = 2^{\alpha} \delta$, where $\delta \in \N$ is odd.  
Then $\left( \sum_{n=0}^{\infty}  q^{ 8\binom{n+1}{2} } \right)^{\mu} 
\equiv  \left( \sum_{n=0}^{\infty}  q^{ 8\cdot 2^{\alpha} \binom{n+1}{2} } \right)^{\delta} $ 
is a monic polynomial in $q^{ 8\cdot 2^{\alpha}  }$ and, as such,  its reciprocal cannot have a term~$q^{ 2^{\alpha} \delta  }$, which is $q^{\mu}$.  Thus
$$
D_0 \equiv 
\Coeff \left(  q^{\mu},     \dfrac{1}{ \prod_{  j \in \N}  (1-q^j)^{24\mu}  }       \right) =0, 
$$
when $\mu \ge 1$ and this completes the proof of Theorem~\ref{antisym}.

\section{When do we have $\Grit(\phi)=\Borch\bigl(-(\phi|V_2)/\phi\bigr)$ ?} 

In order to complete the line of thought that began this research 
and to completely characterize the paramodular forms that are 
both Gritsenko lifts of theta blocks without theta denominator and Borcherds Products, 
it would suffice to prove the following conjecture.  

\begin{conjecture}
\label{wow}
Let $\phi\in J_{k,t}$ be a theta block without theta denominator and with vanishing order one in $q=e(\tau)$.  
 Then $\Grit(\phi)= \Borch(\psi)$ for $\psi=-\frac{\phi|V_2}\phi$. 
\end{conjecture}

We know  in the above conjecture that 
$\Borch(\psi)$ and $\Grit(\phi)$ are both symmetric forms  in $M_k(K(t))$ and that they have identical
first and second Fourier Jacobi coefficients.  
The following theorem proves Conjecture~\ref{wow} for weights~$k$ 
satisfying $4 \le k \le 11$.  
Thus, item~{(4)} of Theorem~\ref{borcherdsproductseverywhere} may be strengthened for these weights 
to assert the complete equality of $\Borch(\psi)$ and $\Grit(\phi)$. 
The proof proceeds by demonstrating 
an exhaustive list  of examples.  

\begin{thm}\label{trivialthbl}
{\rm (Theta-products of order one.)\/}
Let $\ell \in \N$ be in the range  $1\le \ell\le 8$, and let $d_1,\ldots,d_\ell \in \N$
with $(d_1+\dots+d_\ell)\in 2\N$.
Then Conjecture 8.1 is true for the Jacobi form 
$$
\eta^{3(8-\ell)}\vartheta_{d_1}\cdot \ldots \cdot \vartheta_{d_\ell}\in 
J_{k, t}, {\text{ where }} k=12-\ell \text{ and }\ t=(d_1^2+\dots+d_\ell^2)/2.
$$ 
Additionally, 
this product is a Jacobi cusp form if $\ell<8$ or if $\ell=8$
and ${(d_1\cdot \ldots \cdot d_8)}/{d^8}$ is even 
where 
$d=(d_1,\dots, d_8)$ is the greatest common divisor of the $d_i$.
\end{thm}
\begin{proof}
This theorem is a direct corollary of \cite[Theorem 3.2]{GritRefl}
where it was proved  that  
$$
F_{D_8}=
{\rm Lift}\left(\vartheta(\tau,z_1)\cdots \vartheta(\tau,z_8) \right)
=\Borch\bigl(-\frac{(\vartheta(\tau,z_1)\cdots \vartheta(\tau,z_8))|_4V_2}
{\vartheta(\tau,z_1)\cdots \vartheta(\tau,z_8)}\bigr)
$$
for a strongly reflective modular form $F_{D_8}$ of weight $4$
in ten variables with respect to 
$\widetilde{\rm O}^+(U\oplus U\oplus D_8(-1))$. 
Similar identities were also obtained for the Jacobi forms
$\eta^{3(8-\ell)}\vartheta(\tau,z_1)\cdot \ldots \cdot \vartheta(\tau,z_\ell)$
with respect to the lattice $D_\ell$ for $\ell>1$. To prove the theorem
we have to take the restriction of the last identity for 
$(z_1,\dots,z_8)=(d_1z,\dots,d_8z)$.
The fact that the product is a Jacobi cusp form for $\ell=8$ follows from the Fourier
expansion of the Jacobi theta-series, see the proof of \cite[Lemma 1.2]{GritHulek}.

For $\ell=1$  our arguments also work.  
The orthogonal complement of  $D_7$ in $D_8$ is the lattice $D_1=\latt{4}$
of rank $1$. The (quasi) pullback of $F_{D_8}$ to the Siegel upper half plane, 
defined by the lattice $U\oplus U\oplus \latt{4}$, 
gives the reflective cusp form 
$G_{11}=\Grit(\phi_{11,2})\in S_{11}(K(2))$
(see Remark 2 after  \cite[Corollary 3.5]{GritRefl})
discussed in ${\mathbf N=2}$ of section~2 here. 
We conclude by using the fact that 
$F\mymat{\tau}zz{\omega} \mapsto F\mymat{\tau}{dz}{dz}{d^2\omega}$
defines an injective  map $M_k( K(N)) \to M_k( K(Nd^2) )$ that respects
both liftings.  
\end{proof}

Theorem~\ref{trivialthbl} shows that all examples considered in section~2 
for levels $N=1, \dots, 5$ are the first members of  {\it eight} infinite series 
of Gritsenko's liftings with Borcherds product structure of weight
$12-\ell$ for $\ell=1$, $\dots$, $8$.
For example, for $\ell=4$ we have  
$\Grit(\eta^{12}\vartheta_1^4)\in S_8(K(2))$,
$\Grit(\eta^{12}\vartheta_1^2\vartheta_2^2)\in S_8(K(5))$,
$\Grit(\eta^{12}\vartheta_1^3\vartheta_3)\in S_8(K(6))$ and 
$$
\Grit(\eta^{12}\vartheta_{d_1}\dots\vartheta_{d_4})
\in S_8(K((d_1^2+\dots+d_4^2)/2)).
$$

We can also construct a {\it ninth\/} infinite series of such modular forms of weight $3$. 
Let us take the simplest non-trivial  theta blocks, i.e., with 
a single $\eta$ factor in the denominator. These are  the so-called theta-quarks 
(see \cite{GSZ} and \cite[Corollary 3.9]{CG2}); 
for $a,b \in \N$, set 
$$
\theta_{a,b}=\frac{\theta_a\theta_b\theta_{a+b}}{\eta}\in J_{1,a^2+ab+b^2}(\chi_3),   
\qquad \chi_3=\epsilon_\eta^8, 
\qquad \chi_3^3=1.
$$
The theta-quark $\theta_{a,b}$ is a Jacobi cusp form if $a\not\equiv b \mod 3$. 
The following theorem  is a direct corollary of  \cite[Theorem 4.2]{GritRefl}
about the strongly reflective modular form of weight $3$ with respect to 
O$^+(2U\oplus 3A_2(-1))$.  
\begin{thm}\label{thquarks}  
For $a_1,b_1,a_2,b_2,a_3,b_3 \in \N$, we have 
$$
\Grit(\theta_{a_1,b_1}\theta_{a_2,b_2}\theta_{a_3,b_3})=
\Borch(\psi)\in M_3(K(t))
$$
where $t=\sum_{i=1}^3 (a_i^2+a_ib_i+b_i^2)$ and 
 $\psi=-\dfrac{(\theta_{a_1,b_1}\theta_{a_2,b_2}\theta_{a_3,b_3})|V_2}
{\theta_{a_1,b_1}\theta_{a_2,b_2}\theta_{a_3,b_3}}$.
\end{thm}
This example is very interesting because a paramodular cusp form of weight $3$
with respect to $K(t)$ induces a canonical differential form on the moduli space
of $(1,t)$-polarized abelian surfaces, see \cite{Grit2}.   
Therefore the divisor of the modular form 
in this example gives the class of the canonical divisor of the corresponding Siegel modular 
$3$-fold.

In a subsequent publication, we hope to show that the identity proven as the last example of section~2, 
$\Grit(\phi_{2,37})=\Borch(\psi_{2,37})$,
is also a member of an infinite family of identities for Siegel paramodular 
forms of weight $2$.


\end{document}